\documentclass[a4paper,10pt,oneside]{amsart}
\usepackage{setspace}
\usepackage{amsfonts}
\usepackage{amsmath}
\usepackage{amssymb}
\usepackage{mathrsfs} 
\usepackage{color} 
\usepackage[english]{babel}
\usepackage[T1]{fontenc}
\usepackage{amsthm}
\usepackage{booktabs} 
\usepackage{tikz} 
\usetikzlibrary{arrows%
                ,plotmarks
}
\usepackage[disable]
{todonotes}
\input xy 
\xyoption{all}
\usepackage{enumitem} 
\usepackage{hyperref} 
\hypersetup{
colorlinks=
false
} 
\usepackage[scale=0.74]{geometry} 
\usepackage{xfrac}

\theoremstyle{plain}
\newtheorem{thm}{Theorem}
\newtheorem*{thm*}{Theorem}
\newtheorem{propo}[thm]{Proposition}
\theoremstyle{definition}
\newtheorem{lem}[thm]{Lemma}
\newtheorem{coro}[thm]{Corollary}
\newtheorem{rmk}[thm]{Remark}

\theoremstyle{remark}

\newcommand{\de}{\partial}
\newcommand{\PU}{\gr{P}(U)}
\newcommand{\PV}{\gr{P}(V)}
\newcommand{\PC}{\mathbb{P}(\mathscr{C})}
\newcommand{\PE}{\mathbb{P}(\mathscr{E})}
\newcommand{\pp}{\mathbb{P}}
\newcommand{\Pd}[1]{\mathbb{P}({#1})}
\newcommand{\mappa}[2]{{#1} \rightarrow {#2}}
\newcommand{\spazio}{\rule{1 pt}{0 cm}}
\newcommand{\gr}[1]{\mathbf{#1}}
\newcommand{\odi}[1]{\mathcal{O}_{#1}}
\newcommand{\OPU}{\mathcal{O}_{\PU}}

\newcommand{\EXT}{\mathscr{E}\!xt}
\newcommand{\HOM}{\mathscr{H}\!om}
\newcommand{\END}{\mathscr{E}\!nd}
\newcommand{\sm}{\mbox{sm}}
\newcommand{\GRA}{\gra(m,\Lambda^{2}V)}
\newcommand{\ellecan}{-2}

\newcommand{\CC}{\mathscr{C}}
\newcommand{\rosso}[1]{#1}

\renewcommand{\Bbb}{\mathbb}
\renewcommand{\bold}{\mathbf}

\DeclareMathOperator{\Coh}{Coh}
\DeclareMathOperator{\Spec}{Spec}
\DeclareMathOperator{\Hilb}{Hilb}

\DeclareMathOperator{\Cat}{Cat}
\DeclareMathOperator{\Pf}{Pf}
\DeclareMathOperator{\PGL}{PGL}
\DeclareMathOperator{\GL}{GL}
\DeclareMathOperator{\pari}{par}
\DeclareMathOperator{\gra}{\textbf{Gr}}
\DeclareMathOperator{\depth}{depth}
\DeclareMathOperator{\pd}{pd}
\DeclareMathOperator{\id}{I}
\DeclareMathOperator{\im}{Im}
\DeclareMathOperator{\Proj}{Proj}

\DeclareMathOperator{\Ext}{Ext}
\DeclareMathOperator{\Hom}{Hom}
\DeclareMathOperator{\coker}{coker}
\DeclareMathOperator{\Vi}{V}
\DeclareMathOperator{\Hh}{H}
\DeclareMathOperator{\hh}{h}

\DeclareMathOperator{\End}{End}

\DeclareMathOperator{\R}{R}
\DeclareMathOperator{\Sing}{Sing}
\DeclareMathOperator{\codim}{codim}
\DeclareMathOperator{\Sym}{Sym}
\DeclareMathOperator{\Cl}{Cl}
\DeclareMathOperator{\Pic}{Pic}

\newskip\skipppamount
\def\smskip{\vskip\skipppamount \noindent}

\newskip\skippamount
\def\meskip{\vskip\skippamount}

\begin{document}

\title[On the Hilbert scheme of degeneracy loci of $\mappa{\mathcal{O}_{\PV}^{m}}{\Omega^{}_{\PV}(2)}$]{On the Hilbert scheme of degeneracy loci of twisted differential forms}
\author{Fabio Tanturri}
\address{Mathematik und Informatik\\
Universit\"at des Saarlandes\\
Campus E2 4\\
D-66123\\
Saarbr\"ucken\\
Germany}
\email{tanturri@math.uni-sb.de}
\thanks{Supported by the International School for Advanced Studies (SISSA, Trieste). Partially supported by the Research Network Program ``GDRE-GRIFGA'', the ANR project \mbox{GeoLMI}, and by the PRIN 2010/2011 ``Geometria delle variet\`a algebriche''}



\date{\today}

\begin{abstract}
We prove that, for $3 < m < n-1$, the Grassmannian of $m$-dimensional subspaces of the space of skew-symmetric forms over a vector space of dimension $n$ is birational to the Hilbert scheme of the degeneracy loci of $m$ global sections of $\Omega_{\mathbb{P}^{n-1}}(2)$, the twisted cotangent bundle on $\mathbb{P}^{n-1}$. For $3=m<n-1$ and $n$ odd, this Grassmannian is proved to be birational to the set of Veronese surfaces parametrized by the Pfaffians of linear skew-symmetric matrices of order $n$.
\end{abstract}

\maketitle

\section{Introduction}
\rosso{Degeneracy loci of morphisms of the form $\phi:\mappa{\mathcal{O}_{\mathbb{P}^{n-1}}^{m}}
{\Omega_{\mathbb{P}^{n-1}}^{}(2)}$ arise naturally in algebraic geometry 
and have been extensively studied, both classically and from a modern point of view.

Many 
interesting classical varieties can be obtained as such degeneracy loci: 
in 1891}, Castelnuovo \cite{Castelnuovo} considered the case $n=5$ and showed 
that the degeneracy locus of a general morphism 
$\phi:\mappa{\mathcal{O}_{\mathbb{P}^{4}}^{3}}{\Omega^{}_{\mathbb{P}^{4}}(2)}$ 
is the well-known projected Veronese surface in $\mathbb{P}^{4}$.
Few years later, Palatini \cite{PalatiniSistemi,PalatiniComplessi} focused on $\mathbb{P}^{5}$. The case $m=3$ leads to the
elliptic scroll surface of degree six, which was further studied by Fano \cite{Fano}. The case $m=4$
yields a threefold of degree seven which is a scroll over a cubic surface of $\mathbb{P}^{3}$, also
known as Palatini scroll; an interesting conjecture by Peskine states that it is the only smooth threefold in
$\pp^{5}$ not to be quadratically normal.

The case $(m,n)=(4,5)$ gives rise to the famous Segre cubic primal, a threefold in $\pp^{4}$ \rosso{which is proven to be the unique one having exactly ten distinct singular points and fifteen planes. The Segre cubic primal has been thoroughly studied due both to its rich geometry and its connections with interesting moduli spaces.

A more detailed historical account about these degeneracy loci in general and other classical examples can be found, for instance, in \cite{BazanMezzetti,FaenziFania}.}

Let us denote by $X_{\phi}$ the degeneracy locus arising from a morphism $\phi$. As the Hilbert polynomial of $X_{\phi}$ is generically fixed,
we can define $\mathcal{H}$ as the union of the irreducible 
components, in the Hilbert scheme of subschemes of $\mathbb{P}^{n-1}$, containing the 
degeneracy loci arising from general $\phi$'s.

Let $\mathbb{P}^{n-1}\cong \PV$ be the projectivization of an $n$-dimensional vector space $V$. Relying on a nice interpretation due to Ottaviani (\cite[\textsection3.2]{Ottaviani}, cfr.~Sect.~\ref{2.PCPE}), we can
identify a morphism of the form above with a skew-symmetric matrix of linear
forms in $m$ variables, or with an $m$-uple of elements in $\Lambda^{2}V$; moreover, the natural 
$\GL_{m}$-action does not modify its degeneracy locus, \rosso{so we 
get the natural rational map}
\begin{equation}
\label{rhointro}
\rho:
\xymatrix{
\GRA \ar@{-->}[r] & \mathcal{H}
}
\end{equation}
sending $\phi$ to $X_{\phi}$.

In the case ${(m,n)=(3,5)}$, from the results contained in \cite{Castelnuovo}
one can prove that the component of $\mathcal{H}$ containing Veronese
surfaces in $\mathbb{P}^{4}$ is birational to $\gra(3,\Lambda^{2}V)$. A
similar statement holds for the Palatini scrolls in $\mathbb{P}^{5}$: the main result of
\cite{FaniaMezzetti} states that $\rho$ is birational when $(m,n)=(4,6)$.
In the case $(m,n)=(3,6)$, however, it was proved in \cite{BazanMezzetti},
and in fact classically known to Fano \cite{Fano}, that $\rho$ is dominant and
generically $4:1$. \rosso{Other cases have been recently studied in \cite{FaenziFania}.}
%
%

\rosso{Our main result is a complete description of the features of the map $\rho$.}

\begin{thm*}
Let $m, n \in \mathbb{N}$ satisfying $2 < m < n-1$ and let 
\[\rho:
\xymatrix{
\GRA \ar@{-->}[r] & \mathcal{H} 
}\]
be the rational morphism introduced in (\ref{rhointro}), sending the class of a morphism $\phi:\mappa{\mathcal{O}_{\PV}^{m}}
{\Omega^{}_{\PV}(2)}$ to its degeneracy locus $X_{\phi}$, considered as a point in the Hilbert scheme.
\begin{enumerate}[label=\textup{\roman{*}.}, ref=(\roman{*})]
\item If $m \geq 4$ or $(m,n)=(3,5)$, then $\rho$ is birational; in particular, the Hilbert scheme $\mathcal{H}$ 
is irreducible and generically smooth of dimension $m\left(\binom{n}{2}-m\right)$.
\item If $m=3$ and $n \neq 6$, then $\rho$ is generically injective. Moreover
\begin{itemize}
\item[\textup{ii.a.}] if $n$ is odd, $\rho$ is dominant 
on a closed subscheme $\mathcal{H}'$ of $\mathcal{H}$ of codimension ${\frac{n}{8}(n-3)(n-5)}$. The general element of $\mathcal{H}$ is a general projection in $\PV$ of a Veronese surface $v_{\frac{n-1}{2}}(\pp^{2})$, embedded via the complete linear system of curves of degree $\frac{n-1}{2}$; in particular, $\mathcal{H}$ is irreducible. The general element of $\mathcal{H}'$ is a particular projection in $\PV$, obtained using the linear space spanned by the partial derivatives of order $\frac{n-5}{2}$ of a non-degenerate polynomial $G \in \mathbf{k}[y_{0},y_{1},y_{2}]$ of degree $n-3$ as the center of projection;
\item[\textup{ii.b.}] if $n$ is even, $\rho$ is dominant 
on a closed subscheme $\mathcal{H}'$ of $\mathcal{H}$ of codimension ${\frac{3}{8}(n-4)(n-6)}$. The general element of $\mathcal{H}'$ is a projective bundle $\Pd{\mathscr{G}}$ obtained projectivizing a general stable rank-two vector bundle $\mathscr{G}$ on a general plane curve $C$ of degree $\frac{n}{2}$, with determinant $\det(\mathscr{G})=\mathcal{O}_{C}(\frac{n-2}{2})$.
\end{itemize}
\end{enumerate}
\end{thm*}
Part i.~of the Theorem is the content of Theorem \ref{2.teorema} and Corollary \ref{2.corolll}; the general injectivity of $\rho$ will be proved in Theorem \ref{2.geninj}. In the case $m=3$, the codimensions of $\mathcal{H}'$ in $\mathcal{H}$ are computed in Proposition \ref{codimcasosur}; if $n$ is odd, the characterization of the general element of $\mathcal{H}'$ is performed in Theorem \ref{2.characterization}, while the general element of $\mathcal{H}$ is described in Proposition \ref{genelediH}. In the case $n$ even, this was done in \cite{FaenziFania}.

This theorem provides a complete description, showing that the case $(m,n)=(3,6)$ is the unique in which 
$\rho$ is not generically injective. It shows also that, for ${m=3}$, the 
case $n=5$ is the only one in which we have birationality. The 
missing birationality for an odd $n>6$ can be explained by means of the above 
description of ${\im(\rho) \subset \mathcal{H}}$: the general projection of a Veronese surface is not special in the sense of the Theorem, so it is not in the image of $\rho$. \rosso{For small values of $m,n$, this theorem provides another proof of the classical results already known; it also covers the main results of \cite{FaenziFania}}.

The main tool for performing the cohomology computations needed to prove the Theorem is the so-called Kempf-Lascoux-Weyman's method of calculation of syzygies via resolution of singularities; the original idea of Kempf was that the direct image via $q$ of a Koszul complex of a resolution of singularities $q:\mappa{Y}{X}$ can be used to prove results about the defining equations and syzygies of $X$. This method was successfully used by Lascoux in the case of determinantal varieties, and it is developed in full generality in Weyman's book \cite{Weyman}. \rosso{This approach is very convenient because it allows to deal with degeneracy loci with singularities, a case intractable so far. It is more general than the one adopted in \cite{FaenziFania}, which also strongly depended on the parity of $n$}.

The characterization of the general element in $\im(\rho)$, in the case $m=3$ and $n$ odd, is proved making use of Macaulay's Theorem on inverse systems \cite{Macaulay} and apolarity. \rosso{As an interesting secondary result, we develop an improved version of the Macaulay correspondence for plane curves, showing that it can be specialized }to a correspondence between non-degenerate curves and ideals generated by the Pfaffians of a linear skew-symmetric matrix (Proposition \ref{2.corrisp}).

\meskip The structure of the paper is the following: in Sect.~\ref{2.notations}, we introduce some notation, perform some preliminary constructions and prove some basic properties to be used later. We also provide a \rosso{complete} geometric interpretation of the degeneracy loci we are dealing with and we define explicitly the map $\rho$ introduced in (\ref{rhointro}). In Sect.~\ref{2.normalbu} we provide a description of the normal sheaf of a degeneracy locus in $\PV$; this allows us to produce an upper bound for the dimension of the space of its global sections in Sect.~\ref{2.upperbo}, performed by means of the Kempf-Lascoux-Weyman's method. In Sect.~\ref{2.injandbir} we prove the injectivity and birationality of $\rho$. Finally, in Sect.~\ref{2.casom3} we study the case $m=3$, giving a geometric description of the points in $\im(\rho)$ by means of Macaulay's Theorem and apolarity.

\meskip
\rosso{The case $m=2$, not treated here, can also be 
considered, but with different methods; it will 
be studied in a forthcoming paper.
}
\section{Preliminary constructions and first properties}
\label{2.notations}
\subsection{Notation, dimensions and singularities}\rule{1pt}{0pt}
\label{notadimesingu}\smskip
Let $\mathbf{k}$ be an algebraically closed field of characteristic zero and let $m, n \in \mathbb{N}$ such that ${2<m<n-1}$. We will denote by $U, V$ two $\mathbf{k}$-vector spaces of dimensions $m$, $n$ respectively; by $\PU$ and $\PV$ we will mean the projective spaces of their 1-quotients, i.e.~$\Hh^{0}(\PU,\OPU(1))\cong U$
. We set $\{y_{0},\dotsc,y_{m-1}\}$ and $\{x_{0},\dotsc,x_{n-1}\}$ to be the bases of $U$ and $V$ respectively.
\smskip In this paper we focus on the degeneracy locus $X \subset \PV$ of a general morphism of the form $\mappa{U^{*} \otimes \mathcal{O}_{\PV}} {\Omega_{\PV}(2)}$, i.e.~the scheme cut out by the maximal minors of the matrix locally representing the map. As the degeneracy locus is the same for a map and its transposed, we will rather consider the map ${\varphi:\mappa{\mathcal{T}_{\PV}(-2)}{U \otimes \mathcal{O}_{\PV}}}$, with kernel and cokernel $\mathscr{K}, \mathscr{C}$.
\begin{equation}
\label{2.mappakercoker}
\xymatrix{
0 \ar[r] &
\mathscr{K} \ar[r] &
\mathcal{T}_{\PV}(-2) \ar[r]^-{\varphi} &
U \otimes \mathcal{O}_{\PV} \ar[r] &
\mathscr{C} \ar[r] &
0
}
\end{equation}
The sheaf $\mathscr{C}$ is supported on $X$, i.e.~$\mathscr{C}_{x}=0$ if and only if $x \notin X$. If we denote by $i$ the injection $\mappa{X}{\PV}$, we can write $\mathscr{C}=i_{*}\mathscr{L}$ for some sheaf $\mathscr{L}$ on $X$.

More generally, given a morphism, one can define $D_{k}$ to be the subscheme cut out by the minors of order $k+1$ of the matrix locally representing the morphism. The following well-known result give some information on the codimension of degeneracy loci in general.
\begin{thm}
Let $E$ and $F$ be two vector bundles on a projective space, with ranks $e$, $f$ respectively. Let $E^{*}\otimes F$ be globally generated. Then, for a general morphism $\mappa{E}{F}$, the subschemes $D_{k}$ either are empty or have pure codimension $(e-k)(f-k)$. Moreover, we have that $\Sing(D_{k})=D_{k-1}$ \cite[\textsection 4.1]{Banica}.
\end{thm}
Let us come back to $X \subset \PV$. Being $m < n-1$, we have that $D_{m-2}=\Sing(X)$ and $\codim_{\PV}(D_{m-2})=2(n-m+1)$; moreover, $\mathscr{L}$ has rank greater than one exactly in the points in which the corank of $\varphi$ is at least two. Since $D_{m-2}$ is empty if and only if $2(n-m+1) > n-1$, it turns out that
\begin{equation}
\label{2.smoothness}
X \mbox{ is smooth and }\mathscr{L}\mbox{ is a line bundle over X if and only if } n > 2m-3.
\end{equation}
Furthermore, one has
\begin{equation}
\label{2.codimSing}
\codim_{X}(\Sing(X))=n+2-m \geq 3.
\end{equation}
Let us observe that the dimension of $X$ is $m-1$, regardless of the dimension of the ambient space $\PV$.
\subsection{\texorpdfstring{$\PC$}{P(C)} and the Koszul complex}\rule{1pt}{0pt} \label{2.PCPE}\smskip
We refer to \cite[\textsection 3.2]{Ottaviani} for the following interpretation. Let $\varphi^{t}$ be the dual of $\varphi$; a morphism ${\varphi^{t}:\mappa{U^{*}\otimes \mathcal{O}_{\PV}}{\Omega_{\PV}(2)}}$ corresponds to $m$ global sections of $\Omega_{\PV}(2)$. By considering the global sections of the twisted dual Euler sequence
\begin{equation}
\label{2.eulertwisted}
\xymatrix{
0 \ar[r] &
\Omega_{\PV}(2) \ar[r]^-{\iota} &
V \otimes \mathcal{O}_{\PV}(1) \ar[r] &
\mathcal{O}_{\PV}(2) \ar[r] &
0
}
\end{equation}
we may identify $\Hh^{0}(\PV,\Omega_{\PV}(2))$ with $\Lambda^{2}V$, and therefore
\begin{equation*}
\varphi^{t} \in \Hom_{\PV}(U^{*} \otimes \mathcal{O}_{\PV},\Omega_{\PV}(2)) \cong U \otimes \Lambda^{2}V \subset U \otimes V \otimes V.
\end{equation*}
Since the last term is isomorphic to
\begin{equation*}
\Hom_{\mathbf{k}}(V^{*},V \otimes U) \cong \Hom_{\PU}(V^{*}\otimes \mathcal{O}_{\PU},V \otimes \mathcal{O}_{\PU}(1)),
\end{equation*}
the map $\varphi^{t}$ can be regarded also as an $(n \times n)$ matrix $N_{\varphi}$ of linear forms in $y_{0},\dotsc,y_{m-1}$. As $N_{\varphi}$ belongs to $U \otimes \Lambda^{2}V$, it turns out to be skew-symmetric.

If we compose $\varphi^{t}$ with the injection $\iota$ in (\ref{2.eulertwisted}), we get an $(n \times m)$ matrix $M_{\varphi}$ of linear forms in $x_{0},\dotsc,x_{n-1}$; the degeneracy locus does not change, so $X$ can be viewed also as the degeneracy locus of the morphism represented by the matrix $M_{\varphi}$.
\smskip The matrices $M_{\varphi}$ and $N_{\varphi}$ are linked as follows: they represent two different writings of the tensor ${\varphi^{t} \in U \otimes V \otimes V}$, where we consider the projectivization of the first, respectively the second, term. This corresponds to interchanging the roles of columns and variables: in formulas, if $(N_{\varphi})_{i,j}=\sum_{k=0}^{m-1}a_{i,j}^{k}y_{k}$, we get $(M_{\varphi})$ as in (\ref{2.mphi}) below. Therefore, the study of the degeneracy locus of a general $\varphi$ corresponds exactly to the study of the scheme cut out by the maximal minors of a general $(n \times m)$ matrix
\begin{equation}
\label{2.mphi}
M_{\varphi}=\left(
\begin{array}{ccc}
\sum_{i=0}^{n-1}\alpha_{i,0}^{0}x_{i} & \dotso & \sum_{i=0}^{n-1}\alpha_{i,0}^{m-1}x_{i} \\
\vdots & & \vdots\\
\sum_{i=0}^{n-1}\alpha_{i,n-1}^{0}x_{i} & \dotso & \sum_{i=0}^{n-1}\alpha_{i,n-1}^{m-1}x_{i}
\end{array}
\right)
\end{equation}
satisfying $a_{i,j}^{k}=-a_{j,i}^{k}$ for all $i,j,k$.

\meskip Thinking of $\varphi$ as a matrix $N_{\varphi}$ will be useful to provide a geometric interpretation of $X$; for this sake, we fix some notation. Let $\mathscr{E}$ be the cokernel of $N_{\varphi}:\mappa{V^{*}\otimes \mathcal{O}_{\PU}}{V \otimes \mathcal{O}_{\PU}(1)}$ and let $\PE=\Proj \Sym (\mathscr{E})$. The surjection $\mappa{V \otimes \mathcal{O}_{\PU}(1)}{\mathscr{E}}$ turns into an injection of $\PE$ inside ${\mathbb{P}(V \otimes \mathcal{O}_{\PU}(1))}$, which is isomorphic to ${\PU \times \PV}$; we will denote this product by $\mathcal{P}$ for short.
\smskip The same construction can be repeated for $\mathscr{C}$ (or $\mathscr{L}$), and one has $\mathbb{P}(\mathscr{L}) \cong \PC$ as a subscheme of $\mathbb{P}(U \otimes \mathcal{O}_{\PV}) \cong \mathcal{P}$. Let $p, q$ be the projections onto the first and the second factor and $\bar{p}, \bar{q}$ their restrictions to $\PC$; the diagram
\begin{equation}
\label{2.proiezioni}
\xymatrix{
\PU & \mathcal{P} \ar[l]^-{p} \ar[r]^-{q} & \PV \\
& \rule{0pt}{10pt}\PC \ar[ul]^-{\bar{p}} \ar[ur]^-{\bar{q}} \ar@{^{(}->}[u]
}
\end{equation}
commutes. In this situation, we have the canonical surjection
\begin{equation}
\label{2.canonicalsurj}
\xymatrix{
q^{*}(U \otimes \mathcal{O}_{\PV})
\ar[r] &
\mathcal{O}_{\mathbb{P}(U \otimes \mathcal{O}_{\PV})}(1) \cong \mathcal{O}_{\mathcal{P}}(1,0).
}
\end{equation}
The adjunction of direct and inverse image functors gives us an isomorphism
\begin{equation*}
\Hom_{\PV}\left(\mathcal{T}_{\PV}(-2),U \otimes \mathcal{O}_{\PV}\right) \cong \Hom_{\mathcal{P}}\left(q^{*}\mathcal{T}_{\PV}(-2),\mathcal{O}_{\mathcal{P}}(1,0)\right),
\end{equation*}
obtained in one direction by considering the composition of $q^{*}\varphi$ and the surjection (\ref{2.canonicalsurj}), in the other one by applying $q_{*}$. In this way, $\varphi$ can be regarded as a section $s_{\varphi}$ in $\Hh^{0}(\mathcal{P},p^{*}(\mathcal{O}_{\PU}(1))\otimes q^{*}\Omega_{\PV}(2))$, and we may define its zero locus $Y=\Vi(s_{\varphi}) \subset \mathcal{P}$.
\begin{lem}
\label{lemmascambio}
For any $\varphi \in \Hom_{\PV}(\mathcal{T}_{\PV}(2),U \otimes \mathcal{O}_{\PV})$ such that $X \neq \emptyset$, we have
\begin{equation*}
\PC \cong Y \cong \PE.
\end{equation*}
Moreover, $q(Y)=X$ and $p(Y)$ is the support of $\mathscr{E}$.
\begin{proof}
Consider an open subset $\mathcal{U}$ of $\PV$, trivializing $\mathcal{T}_{\PV}(2)$;
its preimage ${\mathcal{U}'=q^{-1}\mathcal{U}}$ is isomorphic to $\mathcal{U} \times \PU$. On the one hand, on $\mathcal{U}'$ the morphism $\varphi$ is represented by a matrix $\varphi_{\mathcal{U}'}$ and the equations describing $\PC \cap \mathcal{U}'$ are determined from the relation
\begin{equation}
\label{relameglioqi}
\nu \cdot \varphi_{\mathcal{U}'}(\mu) = 0,
\end{equation}
where $\nu \in \PU$ and $\mu \in \mathcal{U}$. Indeed, a quotient of $U \otimes \mathcal{O}_{\PV}$ induces a quotient of $\CC$ if and only if its composition with $\varphi$ is zero. On the other hand, imposing the vanishing of $s_{\varphi}$ gives rise to the same condition (\ref{relameglioqi}) on $\mathcal{U}'$.
\smskip This proves the first isomorphism; the same argument holds for the second one.
\end{proof}
\end{lem}
Let $E \boxtimes F$ denote the tensor product $p^{*}E \otimes q^{*} F$ for any pair of sheaves $E$ on $\PU$ and $F$ on $\PV$. The scheme $Y$ is the zero locus of the section $s = s_{\varphi}$ of the vector bundle $\mathcal{O}_{\PU}(1)\boxtimes \Omega_{\PV}(2)$ on $\mathcal{P}$, so we can construct the Koszul complex on $\mathcal{P}$
\begin{equation}
{
\label{2.koszul}
\begin{array}{c}
\xymatrix{
0 \ar[r] &
 \mathcal{O}_{\PU}(1-n) \boxtimes \mathcal{O}_{\PV}(2-n)  \ar[r]^-{\epsilon_{n-1}} &
 \mathcal{O}_{\PU}(2-n) \boxtimes \Omega_{\PV}(4-n) \ar[r]^-{\epsilon_{{n-2}}} &} \\
\xymatrix{\ar[r] & \dotso \ar[r]^-{\epsilon_{2}} &
\mathcal{O}^{}_{\PU}(-1) \boxtimes \Omega^{n-2}_{\PV}(n-2) \ar[r]^-{\epsilon_{1}} &
\mathcal{O}_{\mathcal{P}} \ar[r]^-{\epsilon_{0}} &
\mathcal{O}_{Y} \ar[r] &
0,
}
\end{array}
}
\end{equation}
where we made use of the isomorphisms $\Lambda^{p}\mathcal{T}_{\PV} \cong \Omega^{n-p-1}_{\PV}(n)$. Being $s$ general, this complex is exact.
\subsection{Geometric interpretation of \texorpdfstring{$X$}{X}}\rule{0pt}{0pt}\smskip
Let us focus on $\bar{q}:\mappa{\PC}{X}$, given by the restriction of $q$ as in diagram (\ref{2.proiezioni}). \label{2.geometricint}
\smskip By (\ref{2.smoothness}), if $n > 2m-3$ then $\PC \cong X$ via $\bar{q}$, as $\mathscr{L}$ is a line bundle over $X$. If $X$ is not smooth, then the restriction of $\mathscr{L}$ to the smooth locus $X^{\sm}$ of $X$ is still a line bundle, so we have an isomorphism $\mappa{{\bar{q}}^{-1}(X^{\sm})}{X^{\sm}}$ induced by $\bar{q}$.
\smskip The regular map $\bar{q}$ is not invertible on the subscheme ${Y'}:={\bar{q}}^{-1}(\Sing(X))$. We saw that $\Sing(X)=D_{m-2}$, so the fibers of $\mathscr{C}$ on the general point of $\Sing(X)$ have dimension two. By inequality (\ref{2.codimSing}) we have
\begin{equation}
\label{2.codimrelative}
\codim_{Y}({Y'})=\codim_{X}(\Sing(X))-1 \geq 2.
\end{equation}

We have shown before that $\PC$ may be regarded as the zero locus $Y$ of a general section of a globally generated vector bundle. This implies that $Y$ is smooth for the general choice of $\varphi$. Moreover, $\PC$ can be interpreted also as $\PE$, where $\mathscr{E}$ is the cokernel of a skew-symmetric matrix. We are able to provide a geometric description of such $\PE$, which depends strongly on the parity of $n$.

If $n$ is even, then $N_{\varphi}$ is a skew-symmetric matrix of even order, whose cokernel $\mathscr{E}$ is a rank-two sheaf supported on the hypersurface described by the Pfaffian of $N_{\varphi}$; such hypersurface is singular as soon as $m \geq 7$. The projectivization $\PE$ is then a scroll over (an open subset of) this Pfaffian hypersurface.

If $n$ is odd, $N_{\varphi}$ has odd order and so its determinant is zero; $\mathscr{E}$ is a rank-one sheaf on $\PU$. The locus where $\mathscr{E}$ has higher rank is exactly the subscheme $Z$ defined by the $(n-1) \times (n-1)$ Pfaffians of $N_{\varphi}$. Let $I$ be the ideal of $Z$; for a general $N_{\varphi}$, it satisfies
\[
\pd_{R}(R/I)=\depth(I,R)=\codim_{R}(I)=3,
\] 
being $R=\mathbf{k}[y_{0},\dotsc,y_{m-1}]$. Indeed, the second and the third term always agree (see, for instance, \cite[Theorem 18.7]{Eisenbud}); the first equality is due to Buchsbaum-Eisenbud Structure Theorem \cite{BuchsbaumEisenbudAlg}.

The surjection $\mappa{V \otimes \mathcal{O}_{\PU}}{\mathscr{E}}$ is given by the Pfaffians of $N_{\varphi}$, so ${\mathscr{E}}$ can be identified with ${\mathcal{I}_{Z}(\frac{n-1}{2})}$. Therefore, $\PE$ is the blow-up of $\PU$ along $Z$ (see, for example, \cite[Theorem IV-23]{EisenbudHarris}). Viewed as a subscheme of $\mathcal{P}$, $\PE$ is the closure of the graph of the map given by the $(n-1) \times (n-1)$ Pfaffians of $N_{\varphi}$.

\begin{lem}
\label{2.propX}
The degeneracy locus $X$ is a normal, irreducible variety.
\begin{proof}
Being normal is a local property, but $X$ is locally a general determinantal subscheme, and they are known to be normal. The irreducibility follows from the geometric description just given; when $n$ is even, we observe that the general Pfaffian hypersurface in $\PU$ is irreducible and $X$ is birational to $Y$, which is the closure in $\mathcal{P}$ of a scroll over this hypersurface. When $n$ is odd, $X$ is birational to a blow-up of $\PU$.
\smskip To show that it is a variety, we note that $X$ is of pure dimension, as it has the expected codimension. So it suffices to prove that it is generically smooth, but this follows from (\ref{2.codimSing}).
\end{proof}
\end{lem}

By \cite[Proposition 2.1]{GrusonPeskine}, the dualizing sheaf of $X$ is
\begin{equation}
\label{2.canonicoX}
\omega_{X} = S^{n-m-1}\mathscr{L} \otimes \mathcal{O}_{\PV}(\ellecan),
\end{equation}
where $S^{i}$ denotes the $i$-th symmetric power.
\subsection{Hilbert schemes and Grassmannians}
\label{2.hilbschandgr}
\rule{1pt}{0pt}\smskip
Our aim is to provide a description of the Hilbert scheme of the degeneracy loci arising from $\varphi$, as $\varphi$ varies. For this sake, we define $\mathcal{H}$ to be the union of the irreducible components, in the Hilbert scheme, containing the degeneracy loci $X$ coming from  general choices of $\varphi$.

We have a natural rational map
\[\xymatrix{
\Hom(\mathcal{T}_{\PV}(-2),U \otimes \mathcal{O}_{\PV}) \ar@{-->}[r] & \mathcal{H},}
\]
sending $\varphi$ to the point representing its degeneracy locus. The group $\GL(U)$ induces an action on $\Hom(\mathcal{T}_{\PV}(-2),U \otimes \mathcal{O}_{\PV})$, by multiplication on the left of the matrix $M_{\varphi}$ (\ref{2.mphi}) associated to $\varphi$. The equations cutting out locally the degeneracy locus may change, but the ideal described does not and so the rational map above factors through this action.
\smskip Recall that $\varphi$ can be seen also asz a $(n \times n)$ skew-symmetric matrix $N_{\varphi}$ of linear forms in $\mathbf{k}[y_{0},\dotsc,y_{m-1}]$, or as an $m$-uple of elements in $\Lambda^{2}V$. With this interpretation, an element of $\GL(U)$ acts as a projectivity on these $m$ elements; it does not affect the linear space spanned by them, so the orbit is generically an element of the Grassmannian $\GRA$.

We get the following scenario:
\begin{equation*}
\xymatrix{
\Hom(\mathcal{T}_{\PV}(-2),U \otimes \mathcal{O}_{\PV}) \ar@{-->}[r] \ar@{-->}[d] & \mathcal{H}\\
\GRA \ar@{-->}[ur]_-{\rho}}
\end{equation*}
\rosso{As mentioned in the introduction, the behavior of the map $\rho$ is known in a few cases. The goal of this paper is to prove that the birationality of $\rho$ holds as soon as $m \geq 4$, and to explain why such birationality is missing in the case $m=3$.}


\section{The normal sheaf \texorpdfstring{$\mathcal{N}$}{N}}
\label{2.normalbu}
Within this section, we will show how can the normal sheaf $\mathcal{N}:=\mathcal{N}_{X/\PV}$ be expressed by means of $\mathscr{C}$. This study will provide an upper bound for the dimension of $\mathcal{H}$, thanks to Grothendieck's Theorem (\cite{GrothendieckFondements,HartshorneDeformation}).
\begin{lem}
\label{LReflexive}
The sheaf $\mathscr{L}$, defined in Sect.~\ref{notadimesingu}, is reflexive.
\begin{proof}
Recall that $X$ is normal and integral by Lemma \ref{2.propX}. By \cite[Proposition 1.6]{HartshorneStable}, $\mathscr{L}$ is reflexive if and only if it is torsion-free and normal; a coherent sheaf $\mathcal{F}$ is said to be normal if, for every open set $U \subseteq X$ and every closed subset $Z \subset U$ of codimension at least two, the restriction map $\mappa{\mathcal{F}(U)}{\mathcal{F}(U \setminus Z)}$ is bijective \cite{Barth}.
\smskip The torsion-freeness of $\mathscr{L}$ follows from the fact that $\mathscr{L}$ is a Cohen-Macaulay sheaf. Indeed, let $\mathscr{M}_{x}$ be the maximal ideal of the local ring $\odi{X,x}$. Since $\mathscr{C}|_{D_{k}\setminus D_{k-1}}$ is a vector bundle of rank $m-k$ on $D_{k}\setminus D_{k-1}$ (cfr.~Sect.~\ref{notadimesingu}), by the Auslander-Buchsbaum formula we have $\depth(\mathscr{M}_{x},\mathscr{L}_{x})=\dim \odi{X,x}$ for any $x$, hence $\mathscr{L}_{x}$ is a Cohen-Macaulay module.
\smskip To show that $\mathscr{L}$ is normal, we first observe that $\odi{\CC}(1)$ is reflexive, hence normal itself. If $U$ is an open subset of $X$ and $Z$ a closed subset of $X$ of codimension at least two, then $\bar{q}^{-1}(U)$ is open in $Y$ and $\bar{q}^{-1}(Z)$ is closed of codimension at least two. The conclusion follows since
\[
\xymatrix{
\mathscr{L}(U) = (\odi{\CC}(1))(\bar{q}^{-1}(U))
\ar[r] &
(\odi{\CC}(1))(\bar{q}^{-1}(U\setminus Z))=\mathscr{L}(U\setminus Z)
}
\]
is bijective.
\end{proof}
\end{lem}
\begin{lem}
\label{2.HomCC}
In the settings of Sect.~\ref{2.notations}, we have $\HOM_{X}(\mathscr{L},\mathscr{L})\cong \mathcal{O}_{X}$
.
\begin{proof}
The lemma is trivial when $\mathscr{L}$ is a line bundle, i.e.~when $n>2m-3$. For the general case, we look at the map
\begin{equation}
\label{2.symmetricinjection}
\xymatrix{
\HOM_{X}(\mathscr{L},\mathscr{L}) \ar[r] &
\HOM_{X}(S^{n-m-1}\mathscr{L},S^{n-m-1}\mathscr{L})
}
\end{equation}
given by $f \mapsto f^{n-m-1}$. The term on the right is isomorphic to $\mathcal{O}_{X}$, by (\ref{2.canonicoX}) and since by \cite[Proposition 2.1]{GrusonPeskine} we have
\[
\HOM_{X}(\omega_{X},\omega_{X})\cong \mathcal{O}_{X}.
\]
The sheaf $\mathscr{L}$ is torsion free by Lemma \ref{LReflexive}. The sheaf $\HOM_{X}(\mathscr{L},\mathscr{L})$ is torsion-free too: indeed, it is a subsheaf of the direct sum of $m$ copies of $\mathscr{L}$, as it results by applying $\HOM_{X}(-,\mathscr{L})$ to sequence (\ref{2.mappakercoker}) restricted to $X$. The map (\ref{2.symmetricinjection}) is then a non-zero map between two rank-one torsion-free sheaves, so its kernel vanishes.
\smskip The lemma is proved as soon as we consider the following chain:
\begin{equation*}
\xymatrix{\mathcal{O}_{X} \ar@{^{(}->}[r] &
\HOM_{X}(\mathscr{L},\mathscr{L}) \ar@{^{(}->}[r] &
\HOM_{X}(S^{n-m-1}\mathscr{L},S^{n-m-1}\mathscr{L}) \cong \mathcal{O}_{X}.
} \qedhere
\end{equation*}
\end{proof}
\end{lem}
\begin{propo}
\label{2.inormale}
In the settings of Sect.~\ref{2.notations}, we have $i_{*} \mathcal{N} \cong \EXT^{1}_{\PV}(\mathscr{C},\mathscr{C})$ \rosso{\cite[Lemma 3.5]{FaenziFaniaDeterminantal}}.
\begin{proof}
The normal sheaf can be characterized also via the isomorphism
\begin{equation*}
i_{*}\mathcal{N} \cong \EXT_{\PV}^{1}\left(i_{*}(\mathcal{O}_{X}),i_{*}(\mathcal{O}_{X})\right),
\end{equation*}
so it is sufficient to show that $\EXT^{1}_{\PV}(\mathscr{C},\mathscr{C})$ is isomorphic to the right-hand-side. In the forthcoming Lemma \ref{lemmaseqspettt} we will show that there is a spectral sequence 
\begin{equation*}
E^{p,q}_{2}=\EXT^{p}_{\PV}(i_{*}(\mathcal{O}_{X}),i_{*}(\EXT^{q}_{X}(\mathscr{L},\mathscr{L}))) \Rightarrow \EXT^{p+q}_{\PV}(\mathscr{C},\mathscr{C}).
\end{equation*}
By Lemma \ref{2.HomCC}, the conclusion holds if we show that $\EXT^{1}_{X}(\mathscr{L},\mathscr{L})=0$. By adjunction we get
\begin{equation*}
\EXT^{1}_{X}(\mathscr{L},\mathscr{L}) \cong \EXT^{1}_{Y}({\bar{q}}^{*}(\mathscr{L}),\mathcal{O}_{Y}(1,0)).
\end{equation*}

Recall that $Y \cong \mathbb{P}(\mathscr{L})$, so on $Y$ we have a short exact sequence
\begin{equation*}
\xymatrix{
0 \ar[r] &
\Omega \ar[r] &
{\bar{q}}^{*}\mathscr{L} \ar[r] &
\mathcal{O}_{Y}(1,0) \ar[r] &
0,
}
\end{equation*}
where $\Omega$ is the kernel of the canonical surjection on the right: it may be considered as the relative cotangent sheaf of $\bar{q}$. Moreover, it is supported on ${Y'}$. If we apply the functor $\HOM_{Y}(-,\mathcal{O}_{Y}(1,0))$ to the short exact sequence above, we get
\begin{equation*}
\xymatrix{
\EXT_{Y}^{1}(\Omega,\mathcal{O}_{Y}(1,0)) \ar[r] &
\EXT_{Y}^{1}({\bar{q}}^{*}(\mathscr{L}),\mathcal{O}_{Y}(1,0)) \ar[r] &
0,
}
\end{equation*}
as $\mathcal{O}_{Y}(1,0)$ is a line bundle on $Y$. The first sheaf vanishes since its support, by (\ref{2.codimrelative}), has codimension at least two, so the second one vanishes too.
\end{proof}
\end{propo}
\begin{lem}
\label{lemmaseqspettt}
We have the following cohomological spectral sequence:
\[
E^{p,q}_{2}=\EXT^{p}_{\PV}(i_{*}(\mathcal{O}_{X}),i_{*}(\EXT^{q}_{X}(\mathscr{L},
\mathscr{L}))) \Rightarrow \EXT^{p+q}_{\PV}(\mathscr{C},\mathscr{C}).
\]
\begin{proof}
Let $\mathcal{E}$, $\mathcal{F}$ be two coherent sheaves on $X$ and consider the two functors
\[
\Psi=\HOM_{\PV}(i_{*}(\odi{X}),i_{*}(-)):\xymatrix{\Coh(X) \ar[r] & \Coh(\PV)}
\]
and
\[
\Phi=\HOM_{X}(\mathcal{E},-):\xymatrix{\Coh(X) \ar[r] & \Coh(X)}.
\]
Their composition $\Psi \circ \Phi$ sends $\mathcal{F}$ to
\begin{equation}
\label{isodafaenzi}
\HOM_{\PV}(i_{*}(\odi{X}),i_{*}(\HOM_{X}(\mathcal{E},\mathcal{F}))) \cong \HOM_{\PV}(i_{*}(\mathcal{E}),i_{*}(\mathcal{F})).
\end{equation}
We can see the last isomorphism by working locally on $\Spec(A) 
\subset X$ and on ${\Spec(B) \subset \PV}$, replacing $i$ with the 
closed embedding $\mappa{\Spec(A)}{\Spec(B)}$ induced by a 
surjective map of $\mathbf{k}$-algebras $\mappa{B}{A}$. $
\mathcal{E}$ and $\mathcal{F}$ are locally replaced by finitely 
generated $A$-modules $M$, $N$, which may be regarded as $B$-modules as well. To prove
(\ref{isodafaenzi}) it is sufficient to exhibit an isomorphism
\[
\Hom_{B}(M,N) \cong \Hom_{A}(A,\Hom_{A}(M,N));
\]
for this sake, we consider the $B$-morphism taking $u:\mappa{M}{N}$ to the $A$-morphism taking $1_{A}$ to $u$ regarded as an $A$-morphism. It is straightforward to check that this is indeed an isomorphism.
\smskip The spectral sequence in the statement follows from the Grothendieck's spectral sequence associated to the composition of the two left-exact functors $\Psi \circ \Phi$, applied after replacing both $\mathcal{E}$ and $\mathcal{F}$ with $\mathscr{L}$.
\end{proof}
\end{lem}
%
\section{An upper bound for \texorpdfstring{$\hh^{0}(X,\mathcal{N})$}{h0(X,N)}}
\label{2.upperbo}
The aim of this section is to provide an upper bound for the dimension of $\Hh^{0}(X,\mathcal{N})$. Since we have $\Hh^{0}(X,\mathcal{N})\cong \Hh^{0}(\PV,i_{*}\mathcal{N})$, we can make use of the isomorphism provided by Proposition \ref{2.inormale}.

By Lemma \ref{2.HomCC} and since
\[\HOM_{\PV}(\mathscr{C},\mathscr{C})\cong i_{*}\HOM_{X}(\mathscr{L},\mathscr{L}),\]
we have $\HOM_{\PV}(\mathscr{C},\mathscr{C}) \cong i_{*}\mathcal{O}_{X}$. If we apply $\HOM_{\PV}(-,\mathscr{C})$ to sequence (\ref{2.mappakercoker}), we get the following diagram:
%
\begin{equation}
\label{2.injection}
\xymatrix{
& & & 0 \ar[d] & \\
0 \ar[r] & i_{*}\mathcal{O}_{X} \ar[r] & \mathscr{C}^{m} \ar[r] \ar@{=}[d] & \HOM_{\PV}(\im(\varphi),\mathscr{C}) \ar[r] \ar[d] & i_{*}\mathcal{N} \ar[r] \ar[d] & 0 \\
0 \ar[r] & i_{*}\mathcal{O}_{X} \ar[r] & \mathscr{C}^{m} \ar[r]^-{\psi} & \Omega_{\PV}(2) \otimes \mathscr{C} \ar[r] & \mathscr{F} \ar[r]  & 0 
}
\end{equation}
%
where $\mathscr{F}$ is defined as the cokernel of $\psi$ and $\mathscr{C}^{m}$ replaces $U^{*} \otimes \mathscr{C}$ for short. Via the snake lemma we deduce that the map $\mappa{i_{*}\mathcal{N}}{\mathscr{F}}$ is an injection, providing an upper bound
\begin{equation}
\label{2.inequality}
\hh^{0}(X,\mathcal{N}) \leq \hh^{0}(\PV,\mathscr{F}).
\end{equation}
By computing $\hh^{0}(\PV,\mathscr{F})$ and by Grothendieck's Theorem, we will have an upper bound for the dimension of $\mathcal{H}$.
\subsection{Cohomology computations}\rule{1pt}{0pt}
\label{cohocomputat}\smskip
The main tool to compute the cohomology groups of the second row of diagram (\ref{2.injection}) is the Koszul complex (\ref{2.koszul}). Making use of it, we provide the following lemmas.
\begin{lem}
\label{2.cohomOxt}
The cohomology groups of $\mathcal{O}_{Y}$ are of dimension
\begin{equation*}
\hh^{i}(Y,\mathcal{O}_{Y})=
\left\{
\begin{array}{ll}
1 & \mbox{if } i=0\\
\binom{\frac{n}{2}-1}{\frac{n}{2}-m} & \mbox{if } i=m-2, n \mbox{ even}, n \geq 2m\\
0 & \mbox{otherwise}
\end{array}
\right.
\end{equation*}
\begin{proof}
Recall that, for any pair of sheaves $E$ on $\PU$ and $F$ on $\PV$, the K\"unneth formula holds:
\[
\Hh^{i}(\mathcal{P},E \boxtimes F) \cong \bigoplus_{j=0}^{i}\Hh^{j}(\PU, E)\otimes \Hh^{i-j}(\PV,F).
\]
By means of this and Bott formula, we are able to 
compute the cohomology groups of the $r$-th term in the Koszul complex 
(\ref{2.koszul}). For $1 \leq r \leq n-1$ we get
%
\[
\hh^{i}(\mathcal{P},\mathcal{O}^{}_{\PU}(-r)\boxtimes \Omega^{n-r-1}_{\PV}(n-2r))
=\left\{
\begin{array}{ll}
\binom{\frac{n}{2}-1}{\frac{n}{2}-m} & \mbox{if } n \mbox{ even}, n\geq 2m, i=\frac{n}{2}+m-2, r=\frac{n}{2}\\
0 & \mbox{otherwise}
\end{array}
\right.
\]
%
so there is at most one non-vanishing cohomology group. We have
\begin{align*}
\Hh^{\frac{n}{2}+m-2}(\mathcal{P},\mathcal{O}^{}_{\PU}(-\frac{n}{2})\boxtimes \Omega^{\frac{n}{2}-1}_{\PV}) & \cong  \Hh^{\frac{n}{2}+m-2}(\mathcal{P},\ker(\epsilon_{\frac{n}{2}-1}))\\
& \cong  \Hh^{\frac{n}{2}+m-3}(\mathcal{P},\ker(\epsilon_{\frac{n}{2}-2})) \\
& \cong  \dotso \\
& \cong  \Hh^{m-1}(\mathcal{P},\ker(\epsilon_{0})),
\end{align*}
whence the result, as soon as we consider the cohomology groups of the terms in the short exact sequence
\[
\xymatrix{
0 \ar[r] &
\ker(\epsilon_{0}) \ar[r] &
\mathcal{O}_{\mathcal{P}} \ar[r] &
\mathcal{O}_{Y} \ar[r] & 0.
} \qedhere
\]
\end{proof}
\end{lem}
\begin{lem} \label{2.cohomCm}
The cohomology groups of $\mathcal{O}_{Y}^{m}(1,0)$ are of dimension
\begin{equation*}
{
\hh^{i}(Y,\mathcal{O}_{Y}^{m}(1,0))=
\left\{
\begin{array}{ll}
m^{2} & \mbox{if } i=0
\\
m\binom{\frac{n}{2}-2}{\frac{n}{2}-m-1} & \mbox{if } i=m-2
, n \mbox{ even}, n\geq 2m+2\\
0 & \mbox{otherwise}
\end{array}
\right.}
\end{equation*}
\begin{proof}
The Koszul complex (\ref{2.koszul}) twisted by $\mathcal{O}_{\mathcal{P}}(1,0)$ is a locally free resolution of $\mathcal{O}_{Y}(1,0)$. Again by means of K\"unneth and Bott formulas, we can compute the cohomology of the $r$-th term, $1 \leq r \leq n-1$, in such resolution:
\[
\hh^{i}(\mathcal{P},\mathcal{O}^{}_{\PU}(1-r)\boxtimes \Omega^{n-r-1}_{\PV}(n-2r))=
\left\{
\begin{array}{ll}
\binom{\frac{n}{2}-2}{\frac{n}{2}-m-1} & \mbox{if } n \mbox{ even}, n\geq 2m+2, i=\frac{n}{2}+m-2, r=\frac{n}{2}\\
0 & \mbox{otherwise}
\end{array}
\right.
\]
As in the proof of the previous lemma, we obtain
\[
\Hh^{\frac{n}{2}+m-2}(\mathcal{P},\mathcal{O}^{}_{\PU}(1-\frac{n}{2})\boxtimes \Omega^{\frac{n}{2}-1}_{\PV}) \cong \Hh^{m-1}(\mathcal{P},\ker({\epsilon_{0}}')),
\]
where ${\epsilon_{0}}'$ is the map $\epsilon_{0}$ in the Koszul complex twisted by $\mathcal{O}_{\mathcal{P}}(1,0)$. The result follows by considering the cohomology groups of the short exact sequence
\[
\xymatrix{
0 \ar[r] &
\ker({\epsilon_{0}}') \ar[r] &
\mathcal{O}_{\mathcal{P}}(1,0) \ar[r] &
\mathcal{O}_{Y}(1,0) \ar[r] & 0.
}
\qedhere
\]
\end{proof}
\end{lem}
\begin{lem}
\label{2.cohomOC}
The cohomology groups of $q^{*}\Omega_{\PV}(2) \otimes \mathcal{O}_{Y}(1,0)$ have dimension
%
\[
\hh^{i}(q^{*}\Omega_{\PV}(2) \otimes \mathcal{O}_{Y}(1,0))
=\left\{
\begin{array}{ll}
m\binom{n}{2}-1 & \mbox{if } i=0, m > 3\\
\binom{\frac{n}{2}-1}{\frac{n}{2}-m} & \mbox{if } i=m-3, n \mbox{ even}, n \geq 2m > 6\\
n\binom{\frac{n-3}{2}}{\frac{n-1}{2}-m} & \mbox{if } i=m-3, n \mbox{ odd}, n \geq 2m > 6\\
\frac{1}{8}n(13n-18) & \mbox{if } i=0, n \mbox{ even}, m=3\\
\rule{0pt}{10pt}\frac{1}{8}(n-1)(n^{2}+5n+8) & \mbox{if } i=0, n \mbox{ odd}, m=3\\
0 & \mbox{otherwise}
\end{array}
\right.
\]
%
\begin{proof}
The Koszul complex (\ref{2.koszul}) twisted by $\mathcal{O}_{\PU}(1)\boxtimes \Omega_{\PV}(2)$ is a locally free resolution of the vector bundle ${q^{*}\Omega_{\PV}(2) \otimes \mathcal{O}_{Y}(1,0)}$; let us denote by $\delta_{r}$ its differentials. If
\[\mathcal{G}_{r}:=\Omega_{\PV}^{n-r-1}(n-2r) \otimes \Omega^{}_{\PV}(2),\]
its $r$-th term is ${\mathcal{O}_{\PU}(1-r)} \boxtimes \mathcal{G}_{r}$.

To compute the cohomology groups of $\mathcal{G}_{r}$, we consider the twisted Euler sequence (\ref{2.eulertwisted}), tensored by $\Omega_{\PV}^{n-r-1}(n-2r)$:
\begin{equation}
\label{2.eulerotwistataa}
\xymatrix{
0 \ar[r] &
\mathcal{G}_{r} \ar[r] &
V \otimes \Omega_{\PV}^{n-r-1}(n-2r+1) \ar[r] & 
\Omega_{\PV}^{n-r-1}(n-2r+2) \ar[r] &
0.
}
\end{equation}
For any $1 < r < n-1$, by Bott formula we have
\[
\hh^{i}(\PV,V \otimes \Omega_{\PV}^{n-r-1}(n-2r+1)) = \left\{
\begin{array}{ll}
n & \mbox{if } i=r-2, 2r=n+1\\
0 & \mbox{otherwise}
\end{array}
\right.
\]
\[
\hh^{i}(\PV,\Omega_{\PV}^{n-r-1}(n-2r+2)) = \left\{
\begin{array}{ll}
\binom{n}{2} & \mbox{if } i=0, r=2, n \geq 3\\
1 & \mbox{if } i=r-3>0, 2r=n+2\\
0 & \mbox{otherwise}
\end{array}
\right.
\]
From the long exact sequence induced by (\ref{2.eulerotwistataa}), we get, for any $1 < r < n-1$,
\begin{equation*}
\hh^{i}\left(\PV,\mathcal{G}_{r} \right)=\left\{
\begin{array}{ll}
\binom{n}{2} & \mbox{if } i=1, r=2\\
n & \mbox{if } i=r-2 \geq 1, 2r=n+1\\
1 & \mbox{if } i=r-2 \geq 1, 2r=n+2\\
0 & \mbox{otherwise}
\end{array}
\right.
\end{equation*}
The cohomology groups of $\mathcal{G}_{0}$ and $\mathcal{G}_{n-1}$ can be computed directly via Bott formula. When $r=1$, one has $\mathcal{G}_{1}\cong \END(\mathcal{T}_{\PV})$, for which the only non-vanishing group is $\Hh^{0}(\PV,\END(\mathcal{T}_{\PV}))\cong \mathbf{k}$.

Again by K\"unneth formula, we get
\[
\hh^{i}(\mathcal{P},{\mathcal{O}_{\PU}(1-r)} \boxtimes \mathcal{G}_{r})\!=\!
\left\{
\begin{array}{ll}
\binom{\frac{n-2}{2}}{\frac{n}{2}-m}& \!\mbox{if }
r=\frac{n+2}{2},
i=m+\frac{n-4}{2},
m \leq \frac{n}{2}
\\
n\binom{\frac{n-3}{2}}{\frac{n-1}{2}-m}& \!\mbox{if }
r=\frac{n+1}{2},
i=m+\frac{n-5}{2},
m \leq \frac{n-1}{2}
\\
1 & \!\mbox{if } r=1, i=0\\
m\binom{n}{2} &\!\mbox{if } r=0,i=0\\
0 & \!\mbox{otherwise}
\end{array}
\right.
\]
Let $\pari(n)$ be the parity of $n$, i.e.~$\pari(n)=1$ if $n$ is odd and $0$ otherwise. Fix $\bar{r}:=\frac{n+2-\pari(n)}{2}$. Since $\mathcal{G}_{r}$ has zero cohomology for $r \notin \{0,1,\bar{r}\}$, we have
\begin{align*}
\Hh^{m+\frac{n-4-\pari(n)}{2}}(\mathcal{P},{\mathcal{O}_{\PU}(1-\bar{r})} \boxtimes \mathcal{G}_{\bar{r}}) & \cong \Hh^{m+\frac{n-4-\pari(n)}{2}}(\mathcal{P},\ker(\delta_{\bar{r}-1})) \\
& \cong \Hh^{m+\frac{n-4-\pari(n)}{2}-1}(\mathcal{P},\ker(\delta_{\bar{r}-2})) \\
& \cong \dotso \\
& \cong \Hh^{m-1}(\mathcal{P},\ker(\delta_{1})).
\end{align*}

The next step gives us
\[
\hh^{i}(\mathcal{P},\ker(\delta_{0}))=\left\{
\begin{array}{ll}
1 & \mbox{if } i=0 \\
\binom{\frac{n-2}{2}}{\frac{n}{2}-m} & \mbox{if } i=m-2, n \mbox{ even},n \geq 2m\\
n\binom{\frac{n-3}{2}}{\frac{n-1}{2}-m} & \mbox{if } i=m-2, n \mbox{ odd}, n \geq 2m\\
0 & \mbox{ otherwise}
\end{array}
\right.
\]
whence the result, which follows by taking into account the short exact sequence
\[
\xymatrix{
0 \ar[r] &
\ker(\delta_{0}) \ar[r] &
\mathcal{G}_{0} \ar[r] &
q^{*}\Omega_{\PV}(2) \otimes \mathcal{O}_{Y}(1,0) \ar[r] & 0
.}\qedhere
\]
\end{proof}
\end{lem}
\begin{rmk}
\label{2.usefulcomputations}
The previous lemmas are enough to compute the cohomology groups of the sheaves appearing in the second row of (\ref{2.injection}). Indeed, the direct images via the morphism $q$ of $\mathcal{O}_{Y}$, $\mathcal{O}_{Y}^{m}(1,0)$, and ${q^{*}\Omega_{\PV}(2)\otimes \mathcal{O}_{Y}(1,0)}$ are respectively $\mathcal{O}_{X}$, $\mathscr{C}^{m}$, and $\mathscr{C} \otimes \Omega_{\PV}(2)$.
\smskip As soon as the higher direct images $\R^{i>0}q_{*}(-)$ are zero, one can apply \cite[Exercise III.4.1]{Hartshorne} and get the desired cohomology groups. To show these vanishings, we argue as follows.
\smskip The hypercohomology spectral sequence of the functor $q_{*}$ applied to the Koszul complex (\ref{2.koszul}) degenerates into an Eagon-Northcott complex, which is a locally free resolution of $\mathcal{O}_{Y}$ (cfr.~\cite[\textsection 2]{GrusonPeskine}); this implies $\R^{0}q_{*}(\mathcal{O}_{Y})\cong\mathcal{O}_{X}$ and the vanishing of $\R^{i>0}q_{*}(\mathcal{O}_{Y})$. The same procedure applied to the Koszul complex twisted by $\mathcal{O}_{\mathcal{P}}(1,0)$ gives rise to a Buchsbaum-Rim complex, which turns out to be a locally free resolution of $\mathscr{C}$ (cfr.~again \cite[\textsection 2]{GrusonPeskine}). As before, this implies $\R^{0}q_{*}(\mathcal{O}_{Y}(1,0))\cong\mathscr{C}$ and the vanishing of $\R^{i>0}q_{*}(\mathcal{O}_{Y}(1,0))$. The third set of vanishings follows from this last argument and the projection formula.
\end{rmk}
We are ready to compute the dimension of $\Hh^{0}(\PV,\mathscr{F})$. Since we want to show that $\rho$ is birational, we compare $\hh^{0}(\PV,\mathscr{F})$ to the dimension of $\GRA$.\begin{propo}
\label{2.ugualiodiversi} \rule{1pt}{0pt}
\begin{enumerate}[label=\textup{\roman{*}.}, ref=(\roman{*})]
\item For any $m > 3$ we have $\hh^{0}(\PV,\mathscr{F})=\dim \GRA$.
\item For $m=3$ and $n \geq 5$, we have%
\begin{equation*}
\hh^{0}(\PV,\mathscr{F}) - \dim \gra(3,\Lambda^{2}V)=
\left\{
\begin{array}{ll}
\frac{3}{8}(n-4)(n-6) & \mbox{if } n \geq 6, n \mbox{ even}\\
\frac{1}{8}n(n-3)(n-5) & \mbox{if } n \geq 5, n \mbox{ odd}
\end{array}
\right.
\end{equation*}
and, in particular, $\hh^{0}(\PV,\mathscr{F})=\dim \gra(3,\Lambda^{2}V)$ if $n=5$ or $n=6$.
\end{enumerate}
\begin{proof}
We can compute $\hh^{0}(\PV,\mathscr{F})$ from the second row of diagram (\ref{2.injection}); the cohomology groups are given by Lemmas \ref{2.cohomOxt}, \ref{2.cohomCm} and \ref{2.cohomOC} (cfr.~Remark \ref{2.usefulcomputations}). This computation proves the statement in all cases but $n \geq 8$, $m = 4$ and $n$ even. For the remaining cases the argument is the following: by the forthcoming Lemma \ref{2.elostesso}, if $n > 2m-3$ we have $\hh^{0}(\PV,\mathscr{F})=\hh^{0}(X,\mathcal{N})$; so to conclude it is sufficient to prove the equality $\hh^{0}(X,\mathcal{N})=\dim \gra(m,\Lambda^{2}V)$ for $m=4$, $n$ even and $n \geq 8$, but this has been done in \cite[Theorem 1]{FaenziFania}.
\end{proof}
\end{propo}
\begin{lem}
\label{2.elostesso}
If $X$ is smooth, then $\hh^{k}(\PV,\mathscr{F})=\hh^{k} (X,\mathcal{N})$ for any $k$.
\begin{proof}
If $X$ is smooth, the sheaves $\mathscr{K}$ and $\mathscr{C}$, defined in (\ref{2.mappakercoker}), are vector bundles on $X$. By \cite[Exercise VI.1(6)]{GolubitskyGuillemin}, we have $\mathcal{N} \cong (\left.\mathscr{K}\right|_{X})^{*} \otimes \mathscr{L}$. Applying the functor $\HOM_{X}(-,\mathscr{L})$ to the sequence (\ref{2.mappakercoker}) restricted to $X$, since $\HOM_{X}(\mathscr{L},\mathscr{L})=\mathcal{O}_{X}$ (Lemma \ref{2.HomCC}) and $\EXT^{1}_{X}(\mathscr{L},\mathscr{L})=0$ ($\mathscr{L}$ is a line bundle), one has
\[
\xymatrix{
0 \ar[r] & \mathcal{O}_{X} \ar[r] & \mathscr{L}^{m} \ar[r] & \HOM_{X}(\left.\im(\varphi)\right|_{X},\mathscr{L}) \ar[r] & 0;
}
\]
as $\EXT^{1}_{X}(\left.\im(\varphi)\right|_{X},\mathscr{L})\cong\EXT^{2}_{X}(\mathscr{L},\mathscr{L})=0$, one also has
\[
\xymatrix{
0 \ar[r] &
\HOM_{X}(\left.\im(\varphi)\right|_{X},\mathscr{L}) \ar[r] &
\mathscr{L} \otimes \left.\Omega_{\PV}(2)\right|_{X} \ar[r] &
\mathcal{N} \ar[r] & 0.
}
\]
As these two sequences fit together to the restriction to $X$ of the second row of diagram (\ref{2.injection}), the conclusion follows.
\end{proof}
\end{lem}

This lemma shows that, even though $\hh^{0}(\PV,\mathscr{F})$ provides only an upper bound for $\hh^{0} (X,\mathcal{N})$ (inequality (\ref{2.inequality})), when $X$ is smooth the link between $\mathscr{F}$ and $\mathcal{N}$ is deeper.
\begin{rmk}
As pointed out in Remark \ref{2.usefulcomputations}, the direct image via $q$ of the Koszul complex (\ref{2.koszul}) (respectively, twisted by $\mathcal{O}_{P}(1,0)$) degenerates into a locally free resolution of $\mathcal{O}_{X}$ (respectively, of $\mathscr{C}$). Instead of computing cohomologies on $\mathcal{P}$ as in Lemmas \ref{2.cohomOxt}, \ref{2.cohomCm}, \ref{2.cohomOC}, we could have worked directly on the Eagon-Northcott or the Buchsbaum-Rim complexes on $\PV$.
\end{rmk}
\section{Injectivity and birationality of \texorpdfstring{$\rho$}{rho}}
\label{2.injandbir}
The purpose of this section is to prove the general injectivity and the birationality of $\rho$, which are the main results of this paper.
\begin{thm}
\label{2.geninj}
The map $\rho:
\GRA \dashrightarrow \mathcal{H}$
is injective on its domain of definition for all $(m,n)$ such that $3 \leq m < n-1$, with the unique exception ${(m,n)=(3,6)}$.
\end{thm}
On the one hand, this means that we can identify an open subset of $\GRA$ with an open subset of a subscheme of $\mathcal{H}$; on the other hand, it gives the lower bound
\begin{equation}
\label{2.lowerbound}
\dim \GRA \leq \dim \mathcal{H},
\end{equation}
which will be fundamental in the proof of the birationality of $\rho$ (Theorem \ref{2.teorema}).

\rosso{In order to prove Theorem \ref{2.geninj}, we need some preliminary results.}
\begin{propo}
\label{2.X1X2}
Following the notation of the previous sections, let $X_{1}, X_{2}$ be the degeneracy loci of two morphisms $\varphi_{1}, \varphi_{2}:\mappa{\mathcal{T}_{\PV}(-2)}{U \otimes \mathcal{O}_{\PV}}$; for $j=1,2$ let $\mathscr{C}_{j}=(i_{j})_{*}(\mathscr{L}_{j})=\coker(\varphi_{j})$ and let $\bar{q}_{j}:\mappa{Y_{j}}{X_{j}}$ be the projection on $\PV$, which is an isomorphism when restricted to $Y_{j} \setminus Y'_{j}$. Assume that $(m,n) \in$ 
\begin{equation*}
\{(m,n) \in \mathbb{N} \times \mathbb{N} \mbox{ such that } 3 \leq m < n-1\} \setminus \{(3,6)\}.
\end{equation*}
If $X_{1} = X_{2}$, then $\mathscr{C}_{1} \cong \mathscr{C}_{2}$.
\begin{proof}
Being $X := X_{1} = X_{2}$, we deduce by (\ref{2.canonicoX}) that
\begin{equation}
\label{2.toerase}
S^{n-m-1}\mathscr{L}_{1} \cong S^{n-m-1}\mathscr{L}_{2}.
\end{equation}
Recall that $\mathscr{L}_{j}$ is a line bundle on the smooth locus $X^{\sm}$, whose complement has codimension at least three by (\ref{2.codimSing}). Let $t \in \mathbb{Z}$ be the minimum integer such that $\hh^{0}(X,\mathscr{L}_{j}(t))\neq 0$ and let $D_{j}$ be the closure in $X$ of the zero locus of a general element $\eta_{j} \in \Hh^{0}(X^{\sm},\left.{\mathscr{L}_{j}(t)}\right|_{X^{\sm}})$. Being $X$, $Y$ normal and irreducible (Lemma \ref{2.propX}), we are allowed to consider their divisor class groups. Since $\mathscr{L}_{j}$ is reflexive (Lemma \ref{LReflexive}), it is determined uniquely by the class of $D_{j}$, by $\mathscr{L}_{j}=\mathcal{I}_{D_{j}}^{*}(-t)$. We have
\[
\mathscr{C}_{1} \cong \mathscr{C}_{2} \quad \Leftrightarrow \quad \mathscr{L}_{1} \cong \mathscr{L}_{2} \quad \Leftrightarrow \quad D_{1} \sim D_{2},
\]
where by ${D_{1}}\sim{D_{2}}$ we mean that the two Weil divisors $D_{j}$ are linearly equivalent, i.e.~they represent the same class in $\Cl(X)$. By \cite[Proposition II.6.5]{Hartshorne} it follows that $\Cl(X)\cong \Cl(X^{\sm})$; by (\ref{2.codimrelative}), also $\Cl(Y_{j} \setminus Y'_{j}) \cong \Cl(Y_{j})$. As $\bar{q}_{1}$ is an isomorphism $\mappa{Y_{1}\setminus Y'_{1}}{X^{\sm}}$, we have
\begin{equation}
\label{2.chainisom}
\Cl(X) \cong \Cl(X^{\sm}) \cong \Cl(Y_{1}\setminus Y'_{1}) \cong \Cl(Y_{1}).
\end{equation}

Consider now the Weil divisors $(n-m-1)D_{j}$, seen as the closures in $X$ of the zero loci of the sections $\eta_{j}^{n-m-1} \in \Hh^{0}(X^{\sm},\left.{(S^{n-m-1}\mathscr{L}_{j}(t)})\right|_{X^{\sm}})$. From (\ref{2.toerase}) we deduce that $(n-m-1)D_{1} \sim (n-m-1)D_{2}$; moreover,
\begin{equation*}
\begin{array}{c}
(n-m-1)D_{1} \quad =_{\Cl(X)} \quad (n-m-1)D_{2} \\
\Updownarrow \\
(n-m-1)\left.D_{1}\right|_{X^{\sm}} \quad =_{\Cl(X^{\sm})} \quad (n-m-1)\left.D_{2}\right|_{X^{\sm}}\\
\Updownarrow \\
(n-m-1)\left.({\bar{q}}_{1}^{*}D_{1})\right|_{Y_{1}\setminus Y'_{1}} \quad =_{\Cl(Y_{1}\setminus Y'_{1})} \quad (n-m-1)\left.({\bar{q}}_{1}^{*}D_{2})\right|_{Y_{1}\setminus Y'_{1}} \\
\Updownarrow \\
(n-m-1)({\bar{q}_{1}}^{*}D_{1}) \quad =_{\Cl(Y_{1})} \quad (n-m-1)({\bar{q}_{1}}^{*}D_{2}).
\end{array}
\end{equation*}
Being $Y_{1}$ smooth, one has $\Cl(Y_{1}) \cong \Pic(Y_{1})$. The latter is torsion-free; indeed, if $n$ is odd, $Y$ is a blow-up of $\PU$ (cfr.~Sect.~\ref{2.geometricint}). If $n$ is even, this was proved in \cite[Lemma 3]{FaenziFania} making use of the fact that the Pfaffian hypersurface cut out by $\Pf(N_{\varphi_{1}})$ (cfr.~Sect.~\ref{2.PCPE}) has torsion-free Picard group, for $(m,n)$ in the supposed range.

As $\Pic(Y_{1})$ has no torsion, we can deduce the equality $({\bar{q}_{1}}^{*}D_{1}) =_{\Cl(Y_{1})} ({\bar{q}_{1}}^{*}D_{2})$, which induces by (\ref{2.chainisom}) the desired $D_{1} \sim D_{2}$.
\end{proof}
\end{propo}
\begin{rmk}
In the case $(m,n)=(3,6)$, the last proposition does not guarantee the general injectivity of $\rho$; in this case the Picard group of the hypersurface in $\PU$ cut out by $\Pf(N_{\varphi})$ has torsion. In fact, it was proved in \cite{BazanMezzetti} and classically known to Fano \cite{Fano} that $\rho$ is $4:1$. As the map is finite and dominant, we have an equality between the dimensions of $\GRA$ and $\mathcal{H}$, as further shown in Proposition \ref{2.ugualiodiversi}.
\end{rmk}
\begin{lem}
\label{2.lemmaprimovan}
For all $3 \leq m < n-1$ we have
\[
\hh^{0}(\PV,\im(\varphi))=\hh^{1}(\PV,\im(\varphi))=0.
\]
\begin{proof}
In the notation of the proof of Lemma \ref{2.cohomCm}, we have $q_{*} \ker({\epsilon_{0}}')=\im(\varphi)$. It is sufficient to check the vanishings
\[
\hh^{0}(\mathcal{P},\ker({\epsilon_{0}}'))=\hh^{1}(\mathcal{P},\ker({\epsilon_{0}}'))=0.
\]
In the proof of Lemma \ref{2.cohomCm} we computed that the only possible non-zero cohomology group of $\ker({\epsilon_{0}}')$ is the $(m-1)$-th, hence the conclusion. 
\end{proof}
\end{lem}
\begin{lem}
\label{2.lemmasecondovan}
For all $3 \leq m < n-1$ we have
\[
\hh^{1}(\PV,\mathscr{K} \otimes \Omega_{\PV}(2))=0,
\]
where $\mathscr{K}$ was defined in (\ref{2.mappakercoker}).
\begin{proof}
Adopting the notation of the proof of Lemma \ref{2.cohomOC}, we deduce that ${q_{*} \ker(\delta_{1})=\mathscr{K} \otimes \Omega_{\PV}(2)}$. By the same argument as above, it is sufficient to check the vanishing of $\hh^{1}(\mathcal{P},\ker(\delta_{1}))$. In the proof of Lemma \ref{2.cohomOC} we computed that the only possible non-zero cohomology group of $\ker(\delta_{1})$ is the $(m-1)$-th, hence the conclusion.
\end{proof}
\end{lem}
We are now ready to prove Theorem \ref{2.geninj}\rosso{, along the lines of \cite[Lemma 9]{FaenziFania}. }
\begin{proof}[Proof of Theorem \ref{2.geninj}] Fix the notation as in Proposition \ref{2.X1X2} and suppose that ${X_{1}}$ and ${X_{2}}$ are equal. By Proposition \ref{2.X1X2}, this induces an isomorphism ${\alpha:\mappa{\mathscr{C}_{1}}{\mathscr{C}_{2}}}$. We are in the following scenario
\begin{equation*}
\xymatrix{
0 \ar[r] &
\mathscr{K}_{1} \ar[r] &
\mathcal{T}_{\PV}(-2) \ar@{.>}[d]^-{\exists \, \gamma} \ar[r]^-{\varphi_{1}} &
U \otimes \mathcal{O}_{\PV} \ar@{.>}[d]^-{\exists \, \beta} \ar[r]^-{\pi_{1}} &
\mathscr{C}_{1} \ar[r] \ar[d]^-{\alpha} &
0\\
0 \ar[r] &
\mathscr{K}_{2} \ar[r] &
\mathcal{T}_{\PV}(-2) \ar[r]^-{\varphi_{2}} &
U \otimes \mathcal{O}_{\PV} \ar[r]^-{\pi_{2}} &
\mathscr{C}_{2} \ar[r] &
0
}
\end{equation*}
We want to show that
\begin{itemize}
\item the isomorphism $\alpha$ induces isomorphisms $\beta$ and $\gamma$ such that the diagram above commutes;
\item up to multiply $\alpha$ by a scalar, we may assume that $\gamma$ is the identity map.
\end{itemize}
In this way, we get that $\varphi_{1}$ and $\varphi_{2}$ belong to the same orbit with respect to the action of $\GL(U)$, i.e.~they represent the same point in $\GRA$.

Let us compose $\pi_{1}$ with $\alpha$. In order to show that such a map can be lifted up to $\beta$, we apply the functor $\Hom_{\PV}(U \otimes \mathcal{O}_{\PV},-)$ to the sequence
\[
\xymatrix{
0 \ar[r] &
\im(\varphi_{2}) \ar[r] &
U \otimes \mathcal{O}_{\PV} \ar[r] &
\mathscr{C}_{2} \ar[r] &
0.
}
\]
Since the last term is
\[
\Ext^{1}_{\PV}(U \otimes \mathcal{O}_{\PV},\im(\varphi_{2})) 
\cong U^{*} \otimes \Hh^{1}(\PV,\im(\varphi_{2}))
\]
and its vanishing is guaranteed by Lemma \ref{2.lemmaprimovan}, we get
\[
\xymatrix{
\End_{\PV}(U \otimes \mathcal{O}_{\PV}) \ar[r] &
\Hom_{\PV}(U \otimes \mathcal{O}_{\PV},\mathscr{C}_{2}) \ar[r] &
0.
}
\]
Therefore, we can lift up $\alpha$ to $\beta$; to check that $\beta$ is an isomorphism, we observe that $\ker(\beta)$ is free and its image via $\pi_{1}$ is zero by commutativity, so we have a map $\mappa{\ker(\beta)}{\im(\varphi)}$. By Lemma \ref{2.lemmaprimovan}, this map has to be zero and so $\ker(\beta)$ is trivial.

To lift up $\beta$ to $\gamma$, we apply the functor $\Hom_{\PV}(\mathcal{T}_{\PV}(-2),-)$ to the sequence
\[
\xymatrix{
0 \ar[r] &
\mathscr{K}_{2} \ar[r] &
\mathcal{T}_{\PV}(-2) \ar[r] &
\im(\varphi_{2}) \ar[r] &
0
}
\]
to get
\[
\xymatrix{
\End_{\PV}(\mathcal{T}_{\PV}(-2)) \ar[r] &
\Hom_{\PV}(\mathcal{T}_{\PV}(-2),\im(\varphi_{2})) \ar[r] &
0;
}
\]
indeed, the last term should be
\[
\Ext^{1}_{\PV}(\mathcal{T}^{}_{\PV}(-2),\mathscr{K}^{}_{2}) \cong \Hh^{1}(\PV,\mathscr{K}^{}_{2} \otimes \Omega^{}_{\PV}(2))
\]
and its vanishing is guaranteed by Lemma \ref{2.lemmasecondovan}. Therefore, $\beta$ can be lifted up to $\gamma$.

Let us notice that $\gamma$ is non-zero and so it is a non-zero multiple $\lambda \id$ of the identity map, as $\mathcal{T}_{\PV}(-2)$ is simple. Finally, the conclusion follows as soon as we substitute $\alpha, \beta$ with their multiples $\lambda^{-1} \alpha,\lambda^{-1} \beta$, so we may take $\gamma=\id$.
\end{proof}

\begin{thm}
\label{2.teorema}
The map $\rho$ is birational for all $(m,n)$ such that $4 \leq m < n-1$, and for $(m,n)=(3,5)$.
\begin{proof}
In the supposed range, we have
\begin{align*}
\dim \GRA & \leq \dim \mathcal{H} & {\mbox{(\ref{2.lowerbound})}}\\
& \leq \hh^{0}(\PV,i_{*}\mathcal{N}) & {\mbox{Grothendieck's Theorem}} \\
& \leq \hh^{0}(\PV,\mathscr{F}) & {\mbox{(\ref{2.inequality})}}\\
& = \dim \GRA. & {\mbox{Proposition \ref{2.ugualiodiversi}}}
\end{align*}
It this way we see that $\rho$ is dominant; by Theorem \ref{2.geninj}, $\rho$ is also generically injective, so it is birational.
\end{proof}
\end{thm}

\begin{coro}
\label{2.corolll}
In the hypotheses of Theorem \ref{2.teorema}, $\mathcal{H}$ is irreducible and generically smooth.
\end{coro}


\section{The case \texorpdfstring{$m=3$}{m=3}: surfaces}
\label{2.casom3}
When $m=3$ and $n$ is even, the general element of $\im(\rho)$ is the projectivization of a general stable rank-two vector bundle on a general plane curve $C$ of degree $\frac{n}{2}$, with determinant $\mathcal{O}_{C}(\frac{n-2}{2})$; this description was given in \cite{FaenziFania}. In this section we will discuss the case $n$ odd.

By Theorems \ref{2.geninj} and \ref{2.teorema}, the map $\rho$ is generically injective but not dominant as soon as $n \geq 7$, so we can identify an open subset of $\gra(3,\Lambda^{2}V)$ with an open subset of a subscheme of $\mathcal{H}$. Our aim is to determine its codimension and describe geometrically the points in $\im(\rho)$ and in $\mathcal{H}$, explaining why a general point of $\mathcal{H}$ cannot be obtained as the degeneracy locus of a morphism $\mappa{\mathcal{T}_{\PV}(-2)}{\mathcal{O}_{\PV} \otimes U}$.

\begin{propo}
\label{codimcasosur}
In the case $m=3$, we have $\codim_{\mathcal{H}} \im(\rho)=\frac{1}{8}n(n-3)(n-5)$ if $n$ is odd, and ${\codim_{\mathcal{H}} \im(\rho)=\frac{3}{8}(n-4)(n-6)}$ if $n$ is even.
\begin{proof}
By Lemma \ref{2.elostesso} and Proposition \ref{2.ugualiodiversi}, it suffices to show that $\mathcal{H}$ is generically smooth along $\im(\rho)$. By (\ref{2.smoothness}) $X$ is smooth; hence, $\mathcal{H}$ is smooth at $X$ if $\hh^{1}(X,\mathcal{N})=\hh^{1}(\PV,\mathscr{F})=0$. This can be obtained by considering the second row of diagram (\ref{2.injection}) and by means of Lemmas \ref{2.cohomOxt}, \ref{2.cohomCm} and \ref{2.cohomOC}.
\end{proof}
\end{propo}

From now on, let us fix $m=3$ and let us suppose $n$ is odd, satisfying $n \geq 7$. Note that all the following results hold also in the case $n=5$: see Remark \ref{cason5dopo}.

\subsection{Veronese surfaces in \texorpdfstring{$\PV$}{P(V)}}\spazio\smskip
Firstly we observe that $n$ is always greater than $2m-3=3$, so by (\ref{2.smoothness}) $X$ is smooth; therefore, in the settings of the previous sections, $Y$ and $X$ turn out to be isomorphic via $\bar{q}$.

On the one hand, as we saw in Sect.~\ref{2.geometricint}, $Y$ is the blow-up of $\PU$ along the subscheme cut out by the $(n-1) \times (n-1)$ Pfaffians $(\Pf_{i})$ of $N=N_{\varphi}$, the Pfaffians of the matrices obtained by deleting the $i$-th row and column from $N$; for the general choice of $\varphi$, the ideal generated by these Pfaffians has codimension three and so its associated subscheme is empty.

On the other hand, $X$ is the image of the regular map given by the $\Pf_{i}$'s. Being these Pfaffians forms of degree $\frac{n-1}{2}$, linearly independent for the general $\varphi$, we can complete them to a basis $\{\Pf_{1},\dotsc,\Pf_{n},C_{1},\dotsc,C_{r-n+1}\}$ of $\mathbf{k}[y_{0},y_{1},y_{2}]_{\frac{n-1}{2}}$ and use this complete linear system of curves to embed $\PU$ in $\mathbb{P}^{r}$, where
\[
r=\dim\left(\mathbf{k}[y_{0},y_{1},y_{2}]_{\frac{n-1}{2}}\right)-1=\binom{\frac{n-1}{2}+2}{2}-1.
\]
The variety $X$ can be seen as the projection in $\PV$ of this Veronese surface in $\mathbb{P}^{r}$ with respect to the center spanned by the $C_{i}$'s.
\begin{equation}
\label{2.veroproj}
\xymatrix{
\PU \ar[rrr]^-{[\Pf_{1}:\dotso:\Pf_{n}:C_{1}:\dotso:C_{r-n+1}]} \ar[rrrd]_-{\rule{0pt}{8pt}[\Pf_{1}: \dotso: \Pf_{n}]} && & \mathbb{P}^{r} \ar@{-->}[d] \\
&&& \PV.
}
\end{equation}

However, not every $n$-uple of forms of degree $\frac{n-1}{2}$ is the set of Pfaffians of a matrix $N$, and this is the reason why $\rho$ is not dominant: only Veronese surfaces parametrized by Pfaffians are contained in $\im(\rho)$. In the next subsections we will explore more this phenomenon.
%
%
\subsection{Apolarity and special projections}\rule{1pt}{0pt}\smskip
Let $R$ be the polynomial ring $\Hh^{0}(\PU,\mathcal{O}_{\PU}(1))=\mathbf{k}[y_{0},y_{1},y_{2}]$. Let $S$ be the polynomial ring of homogeneous differential operators $\mathbf{k}[\de_{0},\de_{1},\de_{2}]$; $R$ acts on $S$ (and conversely) by differentiation:
\[
y^{\alpha}(\de^{\beta})=\alpha ! \binom{\beta}{\alpha}\de^{\beta-\alpha}
\]
if $\beta \geq \alpha$ and $0$ otherwise. Here $\alpha$ and $\beta$ are multi-indices, 
$\alpha !=\prod \alpha_{i}!$, $|\alpha|=\sum \alpha_{i}$, $\binom{\beta}{\alpha}=\prod \binom{\beta_{i}}{\alpha_{i}}$ 
and $\beta \geq \alpha$ if and only if $\beta_{i} \geq \alpha_{i}$ for all $i$. The perfect pairing between
forms of degree $d$ and homogeneous differential operators of the same degree is known as \emph{apolarity}.
\begin{thm}
\label{2.characterization}
Let $G \in R$ be a non-degenerate form of degree $n-3$. Consider a Veronese surface embedded via $|\mathcal{O}_{\PU}(\frac{n-1}{2})|$ in $\mathbb{P}^{r}$, where ${r=\binom{\frac{n-1}{2}+2}{2}-1}$; then its projection $X$ in $\PV$ with respect to the center spanned by ${\{\de^{\alpha}(G)\}_{|\alpha|=\frac{n-5}{2}}}$ is contained in $\im(\rho)$.
\smskip Conversely, a general element of $\im(\rho)$ arises as such a projection.
\end{thm}

Recall that a form $G$ of even degree $k$ is said to be \emph{non-degenerate} if its catalecticant matrix $\Cat(G)$ has maximal rank or, equivalently, if the elements $\{\de^{\alpha}(G)\}_{|\alpha|=\frac{k}{2}}$ are linearly independent in the vector space $R_{\frac{k}{2}}$; $\Cat(G)$ is defined as the square matrix whose $(i,j)$-th element is $D_{i}(D_{j}(G))$, having fixed a basis $\{D_{i}\}$ of $S_{\frac{k}{2}}$.

In order to prove the last theorem, we need some preliminary results. Let $I\subset R$ be an ideal such that $R/I$ is an Artinian, Gorenstein ring with (one-dimensional) socle in degree $k$; as $\Hilb(R/I,k)=1$, there is a homogeneous differential operator $F \in S$ of degree $k$, determined up to scalar, satisfying $G(F)=0$ for any $G \in I$. The operator $F$ is usually called the \emph{dual socle generator}.

Conversely, being given a form $F \in S$ of degree $k$, we can define $F^{\perp}$ as the (homogeneous, irrelevant) ideal in $R$ whose elements $G$ satisfy the property ${G(F)=0}$. The ring $R/F^{\perp}$ is usually denoted by $A^{F}$. The ideal $F^{\perp}$ can be described in terms of the derivatives of $F$, as follows.
\begin{lem}
\label{2.idealeort}
Let $F \in S$ of degree $k$. For any $d \leq k$, the homogeneous component $F^{\perp}_{d}=F^{\perp} \cap R_{d}$ is the orthogonal 
complement of the space of partial derivatives of order $k-d$ of $F$.
\begin{proof}
By convention, the (dual of the) orthogonal complement of a subspace of $S_{d}$ is made up by the differential operators in $R_{d}$ which annihilate all the elements in the subspace. We have therefore to show that, for all $D \in R_{d}$,
\[
D(F)=0
\qquad \Longleftrightarrow \qquad
D(y^{\alpha}(F))=0 \qquad \forall \; |\alpha|=k-d.
\]
Firstly we remark that by apolarity, for a form $F' \in S$ of degree $k-d$, one has
\[
y^{\alpha}(F')=0 \qquad \forall \; |\alpha|=k-d \qquad \Longleftrightarrow
\qquad F'=0.
\]
Consider now $D \in R$ of degree $d$. Since $D(y^{\alpha}(F))=y^{\alpha}(D(F))$, it is enough to apply 
the previous remark to $F'=D(F)$.
\end{proof}
\end{lem}
The two correspondences described above are inverse to each other by the following Theorem on inverse systems by Macaulay, which we recall in the special case of plane curves.
\begin{thm*}[{\cite{Macaulay}}]
The map $F \mapsto A^{F}$ gives a bijection between plane curves $\Vi(F)$, $F \in S$ of degree $k$ and Artinian graded Gorenstein quotient rings of $R$ with socle in degree $k$.
\end{thm*}
For the general matrix $N$, the ideal $I$ generated by the $n$ Pfaffians of order $n-1$ and degree $\frac{n-1}{2}$ has codimension three, and $R/I$ can be easily shown to be an Artinian graded Gorenstein ring with socle in degree $n-3$. In this case, Macaulay correspondence can be rewritten by means of Buchsbaum-Eisenbud Structure Theorem \cite{BuchsbaumEisenbudAlg}, 
linking homogeneous polynomials in $S$ with skew-symmetric matrices of forms on $\PU$. Moreover, if we focus 
only on non-degenerate polynomials, the correspondence restricts to linear skew-symmetric matrices.
\begin{propo}
\label{2.corrisp}
\spazio
\begin{enumerate}[label=\textup{\roman{*}.}, ref=(\roman{*})]
\item The map $F \mapsto F^{\perp}$ gives a bijection between polynomials $F \in S$ of degree 
$n-3$, up to scalars, with $n\geq 5$ odd, and (Artinian graded Gorenstein) ideals $I$ of codimension three in $R$, with socle in 
degree $n-3$, generated by the Pfaffians of a skew-symmetric matrix of forms of positive degrees in $R$.
\item This correspondence restricts to a one-to-one correspondence between non-degenerate polynomials 
$F \in S$ of degree $n-3$, up to scalars, with $n\geq 5$ odd, and (Artinian graded Gorenstein) ideals $I$ of codimension three in $R$ 
generated in degree $\frac{n-1}{2}$ by the $n$ Pfaffians of a $n\times n$ skew-symmetric matrix of 
linear forms in $R$.
\end{enumerate}
\begin{proof} \spazio
\begin{enumerate}[label=\textup{\roman{*}.}, ref=(\roman{*})]
\item By Macaulay correspondence, $A^{F}=R/F^{\perp}$ is an Artinian graded Gorenstein ring. Being 
Artinian, $F^{\perp}$ is irrelevant and so it has codimension three; we can therefore apply Buchsbaum-Eisenbud Structure Theorem and conclude.
\smskip Conversely, an ideal $I$ satisfying the hypotheses 
has codimension three in $R=\mathbf{k}[y_{0},y_{1},y_{2}]$, so it is irrelevant and therefore $R/I$ is 
an Artinian graded Gorenstein ring with socle in degree $n-3$. We conclude again by Macaulay correspondence.
\item Let $F \in S$ be a non-degenerate form of degree $n-3$ and let $I=F^{\perp}$ its Gorenstein, 
codimension-three associated ideal. Let us set $h=\frac{n-3}{2}$ for simplicity. The partial derivatives of order $h$ of $F$ span the 
whole space $S_{h}$; 
therefore, by Lemma \ref{2.idealeort}, $I$ is zero in degree $\leq h$. Moreover, a 
computation shows that ${\dim I_{h+1}=n}$. Let $\nu$ be the minimal number of generators 
of $I$ (hence $\nu \geq n$). By Buchsbaum-Eisenbud Structure Theorem, $I$ is generated by the 
$\nu$ Pfaffians of a ${\nu \times \nu}$ skew-symmetric matrix of homogeneous forms of degree at least 
one. Therefore, the minimum among the degrees of the generators is $\frac{\nu-1}{2}$, but $I$ is non-zero
in degree $h+1=\frac{n-1}{2}$, so $\nu=n$ and the entries of the matrix are linear forms.
\smskip Conversely, let $I$ satisfy the hypotheses of the statement and let us consider the graded Betti
numbers $\beta_{ij}(R/I)$ of the corresponding quotient ring. Being $I$ Gorenstein and minimally
generated by $n$ elements of degree $h+1$, the Betti numbers are all zero with the exceptions
\begin{equation*}
\beta_{0,0}(R/I)=\beta_{3,n}(R/I)=1, \qquad \qquad \beta_{1,\frac{n-1}{2}}(R/I)=\beta_{2,\frac{n+1}{2}}(R/I)=n.
\end{equation*}
One can show by computations that $\Hilb(R/I,n-3)=1$ and $\Hilb(R/I,n-2)=0$, so that the socle is in degree $n-3$. Let $F$ be the dual socle generator; by Macaulay correspondence,
$I=F^{\perp}$. If $F$ was degenerate, then by definition its derivatives of order $h$ 
would be linearly dependent, i.e.~they would not span the whole vector space $S_{h}$. 
But this would imply, by Lemma \ref{2.idealeort}, that $I$ is non-zero in degree $h$, 
hence a contradiction. \qedhere
\end{enumerate}
\end{proof}
\end{propo}

\begin{rmk}
A particular version ($n=7$) of the second correspondence above was already known and, actually, extensively used. 
The correspondence between non-degenerate plane quartics and nets of alternating forms on a vector 
space of dimension seven plays an important role, for instance, in the geometric realizations of 
prime Fano threefolds of genus twelve \cite{MukaiFano, MukaiCurves, SchreyerGeometry}.
\end{rmk}

\begin{rmk}
\label{2.everysystem}
Fixed a Gorenstein, codimension-three ideal $I$ generated by $n$ forms of degree $\frac{n-1}{2}$, 
Buchsbaum-Eisenbud Structure Theorem guarantees the existence of a $n \times n$ skew-symmetric 
matrix $N$ of linear forms whose Pfaffians generate $I$, as we showed in the proof of Proposition 
\ref{2.corrisp}. Actually, \emph{any} minimal system of generators of $I$ arises from a suitable 
matrix $N'$, congruent to $N$. Indeed, consider the matrix $A \in \GL_{n}$ taking the ``Pfaffian'' 
system of generators into the new one. Then these new generators are the Pfaffians of the matrix 
$(A^{-1})^{t}NA^{-1}$.
\end{rmk}
\begin{rmk}
Let us observe that the correspondence developed in Proposition \ref{2.corrisp} is constructive. 
On the one hand, it is clear how, from a skew-symmetric matrix, one can get $F$ by apolarity; on 
the other hand, once given $F$, it is possible to explicitly realize a skew-symmetric matrix whose 
Pfaffians generate the ideal $F^{\perp}$. This is possible thanks to the constructive proof of 
Buchsbaum-Eisenbud Structure Theorem; a concrete example of such a construction can be found in \cite{Tanturri}.
\end{rmk}

We are ready to provide the
\begin{proof}[Proof of Theorem \ref{2.characterization}]
Let $G=\sum c_{\beta} y^{\beta}$. As the projection is linear, the composition ${\mappa{\PU}{\PV}}$ as in (\ref{2.veroproj}) is given by $n$ forms of degree $\frac{n-1}{2}$, whose orthogonal complement in $R_{\frac{n-1}{2}}$ is spanned by the elements $\{\de^{\alpha}(G)\}_{|\alpha|=\frac{n-5}{2}}$. Let us denote by $I$ the ideal generated by these $n$ forms. By Proposition \ref{2.corrisp} and Lemma \ref{2.idealeort} applied to $F:=\sum c_{\beta} \de^{\beta}$, $I=F^{\perp}$ is Gorenstein and has codimension three; by Remark \ref{2.everysystem}, any set of generators of $I$ is made up by the Pfaffians of a suitable matrix $N$, i.e.~any possible projection $X$ is in $\im(\rho)$.

Conversely, consider the image $X$ of $\PU$ via the map given by the $n$ Pfaffians $(\Pf_{i})$ of a general matrix $N$. Let $I$ be the ideal generated by these Pfaffians. $I$ is generically of codimension three, so Proposition \ref{2.corrisp} applies and we get $I=F^{\perp}$ for some non-degenerate $F=\sum c_{\beta} \de^{\beta} \in S$. By Lemma \ref{2.idealeort} we can complete the set of Pfaffians to a basis $\mathcal{B}$ of $R_{\frac{n-1}{2}}$ with the derivatives of order $\frac{n-5}{2}$ of $G:=\sum c_{\beta} y^{\beta}$. Consider $\PU$ embedded in $\mathbb{P}^{r}$ via $\mathcal{B}$ and then projected via $\pi$ to $\PV$ with respect to the center spanned by $\{\de^{\alpha}(G)\}_{|\alpha|=\frac{n-5}{2}}$. The so-obtained Veronese surface $X' \subset \PV$ is in $\im(\rho)$ by the first part of the statement, so it is the image of $\PU$ via a map $[f_{1}:\dotsc:f_{n}]$ given by the Pfaffians of a suitable matrix.
\[
\xymatrix{
&& \mathbb{P}^{r} \ar@{-->}[lldd]_-{\pi}\\
&& \PU \ar[u]^-{\mathcal{B}} \ar[lld]^-{\rule{0pt}{6pt}[f_{1}:\dotso:f_{n}]} \ar[rrd]^-{\rule{6pt}{0pt}[\Pf_{1}:\dotso:\Pf_{n}]}\\
\PV \ar[rrrr]_{\exists \, A} &&&&\PV
}
\]
Since the polynomials $\{f_{i}\}$ and $\{\Pf_{i}\}$ generate the same ideal, there exists an ${A \in \PGL(V)}$ such that the diagram above commutes. It follows that $X$ can be obtained as the projection via $A\circ \pi$ of $\PU$ embedded via $\mathcal{B}$ in $\mathbb{P}^{r}$.
\end{proof}

\subsection{The general element of \texorpdfstring{$\mathcal{H}$}{H}}\spazio\smskip
Theorem \ref{2.characterization} provided a description of the general point in $\im(\rho)$; in particular, a general projection in $\PV$ of the Veronese surface $v_{\frac{n-1}{2}}(\PU)$ does not belong to $\im(\rho)$. Such projections are obviously contained in $\mathcal{H}$, so a natural question is whether they are dense in $\mathcal{H}$.

\begin{propo}
\label{genelediH}
$\mathcal{H}$ is irreducible; its general element is a general projection in $\PV$ of a Veronese surface $v_{\frac{n-1}{2}}(\PU)\subset \pp^{r}$, where $r=\binom{\frac{n-1}{2}+2}{2}-1$.
\end{propo}

To prove this proposition, we consider a parametrization of such projections. The linear space $\mathbf{k}[y_{0},y_{1},y_{2}]_{\frac{n-1}{2}}$ has dimension $r+1$, so we have a rational map
\begin{equation}
\label{mappaxi}
\xymatrix{
\mathbb{A}^{(r+1)n} \ar@{-->}[rr]^-{\xi} && \mathcal{H}
}
\end{equation}
sending $n$ linearly independent forms $f_{1},\dotsc,f_{n}$ of degree $\frac{n-1}{2}$ to the point representing the image of the map
\begin{equation}
\label{imageofthemap}
\xymatrix{
\PU \ar[rr]^-{[f_{1}:\dotso:f_{n}]} && \PV.
}
\end{equation}

From the irreducibility of $\mathbb{A}^{(r+1)n}$ we deduce that $\im(\xi)$ is irreducible.

\begin{lem}
\label{ledimsonogug}
We have $\dim (\im(\xi)) = \dim(\mathcal{H})$.
\begin{proof}
On the one hand, there is a natural $\GL_{3}$-action on $\mathbf{k}[y_{0},y_{1},y_{2}]_{1}$, acting as a change of basis on $U$; this induces an action on $\mathbf{k}[y_{0},y_{1},y_{2}]_{\frac{n-1}{2}}$ and therefore on $\mathbb{A}^{(r+1)n}$, and one can see that $\xi$ 
factors through this action. On the other hand, take two points $V_{1},V_{2}$ in $\im(\xi)$ such that $V_{1}=V_{2}$.
By the commutativity of the diagram
\[
\xymatrix{
& & V_{1} \ar@{=}[d]\\
\PU \ar[urr]^-{[f_{1}:\dotso:f_{n}]\,}_-{\sim} \ar[rr]_-{[g_{1}:\dotso:g_{n}]\,}^-{\sim} & & V_{2}
}
\]
we get an automorphism of $\PU$, i.e.~the two maps $[f_{1}:\dotso:f_{n}]$ and $[g_{1}:\dotso:g_{n}]$ belong to the
same class modulo $\GL_{3}$. Hence
\begin{align*}
\dim (\im(\xi))	&	=\dim (\mathbb{A}^{(r+1)n}) - \dim (\GL_{3}) \\
				&	= n\binom{\frac{n-1}{2}+2}{2} - 9 \\
				&	= \frac{1}{8}n(n+3)(n+1)-9. 
\end{align*}
By Proposition \ref{codimcasosur},
\begin{align*}
\dim (\mathcal{H})	&	=\dim (\im(\rho)) + \codim_{\mathcal{H}} (\im(\rho)) \\
					&	= 3\binom{n}{2}-9 + \frac{1}{8}n(n-3)(n-5) \\
					&	= \frac{1}{8}n(n+3)(n+1)-9
\end{align*}
and therefore the conclusion follows.
\end{proof}
\end{lem}

\begin{proof}[Proof of Proposition \ref{genelediH}]
From Lemma \ref{ledimsonogug} we deduce that the closure of $\im(\xi)$ in $\mathcal{H}$ is an irreducible component of $\mathcal{H}$.
As $\mathcal{H}$ is generically smooth along $\im(\rho)$ (cfr.~Proposition \ref{codimcasosur}), $\im(\rho)$ is contained in only one irreducible component of the Hilbert scheme, namely $\overline{\im(\xi)}$. But $\mathcal{H}$ was defined as the union of the irreducible components containing $\im(\rho)$, so it turns out that $\mathcal{H}=\overline{\im(\xi)}$ and this concludes the proof.
\end{proof}

\begin{rmk}
\label{cason5dopo}
Let us remark that the statement of Theorem \ref{2.characterization} makes perfectly sense also when $n=5$. In this case,
a general element of $\im(\rho)$ is a projection in $\pp^{4}$ of a Veronese surface in $\pp^{5}$, and there is no distinction
between general projections and special projections as those arising in the statement. In other words, any general projection of the Veronese surface in $\pp^{4}$ is in $\im(\rho)$.
\smskip In the proof of Proposition \ref{genelediH} we saw that $\mathcal{H}=\overline{\im(\xi)}$, so we get that $\rho$ is dominant. This, together with the general injectivity, agrees with the birationality of $\rho$ proved in Theorem \ref{2.teorema}.
\end{rmk}



\end{document}